\documentclass[10.5pt,twoside]{article}
\usepackage[plainpages=false]{hyperref}
\usepackage{amsfonts,latexsym,rawfonts,amsmath,amssymb}
\usepackage{amsmath,amssymb,amsfonts,latexsym,lscape,rawfonts}
\usepackage{color}
\newenvironment{proof}[1][Proof]{\begin{trivlist}
\item[\hskip \labelsep {\bfseries #1}]}{\end{trivlist}}

\newcommand{\beqs}{\begin{eqnarray*}}
\newcommand{\eeqs}{\end{eqnarray*}}
\newcommand{\beqn}{\begin{eqnarray}}
\newcommand{\eeqn}{\end{eqnarray}}
\newcommand{\beqa}{\begin{array}}
\newcommand{\eeqa}{\end{array}}

\def\lra{\longrightarrow}

\def\p{\prime}

\def\bc{\begin{center}}
\def\ec{\end{center}}

\def\p{\partial}
\def\b{\bar}

\def\cH{{\cal H}}

\def\tr{\rm tr}

\def\cH{{\mathcal H}}

\newtheorem{prop}{Proposition}[section]
\newtheorem{theo}[prop]{Theorem}
\newtheorem{lem}[prop]{Lemma}

\newtheorem{cor}[prop]{Corollary}
\newtheorem{rem}[prop]{Remark}

\newtheorem{defi}[prop]{Definition}
\newtheorem{conj}[prop]{Conjecture}
\newtheorem{q}[prop]{Question}
\def\begeq{\begin{equation}}
\def\endeq{\end{equation}}
\def\and{\quad{\rm and}\quad}

\let\lra=\longrightarrow

\def\mapright\#1{\,\smash{\mathop{\lra}\limits^{\#1}}\,}

\begin{document}

\title{On the existence of constant scalar curvature K\"ahler metric: a new perspective}
\author{Xiuxiong Chen\\ In memory of Prof. Weiyue Ding}

\date{\today}
\maketitle

\section{Introduction}

In 1950s, E. Calabi first proposed to study the constant scalar curvature K\"ahler (cscK) metric problems.  His ideal is to find the best canonical metric in each given K\"ahler class
(cf. \cite{calabi82, calabi85}), which results in a 4th order, fully nonlinear partial differential equation.    When the first Chern class has a definite sign (positive, negative or zero),
the cscK metric in the suitable multiple of the first Chern class reduces to a K\"ahler-Einstein metrics.  Calabi's program on  K\"ahler-Einstein metrics   is the center of the field for the last few decades where all efforts and techniques
of many mathematicians are devoted to, leading to the final resolution of this difficult problem.  With the existence problem of K\"ahler-Einstein metric settled eventually, perhaps  it is time
to discuss how to attack Calabi's original problem in full generality.  In this note, we propose a ``new, " continuity path in a given K\"ahler class to solve the cscK metric problem.
Module out the profound difficulty in analysis, we hope this will shed light on the existence  problem from direct PDE approach.   This will largely be an expository paper where we concentrate on explaining
various aspects of this new path, except Theorem \ref{openness} in which we proved \emph{openness} along this proposed path. Perhaps more intriguely,  while we can not prove
that the K-energy functional is cocercive in terms of geodesic distance if there exists a cscK metric, we can prove that for any $t\in (0,1)$, this family of twisted K-energy is
indeed coercive in terms of geodesic distance (cf. Theroem \ref{thm3-4} for more precise statement). While we wish this path were new,  as usual, we find ``footsteps of others'', in particular the works
of Fine, Stoppa and Lejmi-Sz$\acute{\text{e}}$kelyhidi \cite{Fine04, Stoppa09, Szekelyhidi14}. The openness theorem is inspired by the renown work of LeBrun-Simanca \cite{LeBrun-Simanca} on
deformation of extremal K\"ahler metrics.

\subsection{A brief account on K\"ahler-Einstein problems}

  In 1976, S.-T. Yau solved the famous
Calabi conjecture  (by showing the existence of Ricci flat K\"ahler metric)
if $C_1 = 0.\;$  Around the same time, S.-T. Yau and T. Aubin
independently solved the existence of the K\"ahler-Einstein metric if $C_1 < 0.\;$   In 1990, S.-T. Yau first suggested that the existence of K\"ahler-Einstein metric is related to certain notion
of stability of the underlying polarization.  
In 1997,   G. Tian introduced the so-called \emph{K-stability} (via \emph{special degeneration}) and
 showed that the existence of K\"ahler-Einstein metric necessarily implies the
K-stability of the underlying polarization through special degeneration.    In 2002, S. K. Donaldson reformulated
it into a notion of \emph{algebraic K-stability}.  In 2012,  with S. K. Donaldson and S. Sun, we  proved 

\begin{theo} [Chen-Donaldson-Sun\cite{cds12-1,cds12-2,cds12-3}] K-stable Fano manifolds admit K\"ahler-Einstein metrics.
\end{theo}
The converse part of the above theorem is due to work of G. Tian\cite{tian97}, S. K. Donaldson\cite{Dona02}, Stoppa \cite{stopp08}, and the most general form is due to R. Berman \cite{Berm12-05}. There is also a Ricci flow approach to
attack the existence problem of the K\"ahler-Einstein metrics. The fundamental problem in the K\"ahler Ricci flow is the so-called \emph{Hamilton-Tian conjecture}:  for any sequence of time $t_i\rightarrow \infty$,
the corresponding sequence of K\"ahler metrics converge in Gromove-Hausdorff sense to a K\"ahler Ricci soliton with at most codimension 4 singularities.
In 2014,  Chen and Wang give an affirmative answer to this conjecture.

\begin{theo} [Chen-Wang\cite{chenwang14-1}] For any Fano manifold, the K\"ahler
Ricci flow will always converge to  a K\"ahler-Ricci soliton
with at most codimension 4 singularities in the sense of
Cheeger-Gromov.  Moreover, the complex structure may jump in the
limit.
\end{theo}

While there are many following up problems related to the existence
problem of K\"ahler-Einstein metrics,  the central problem in K\"ahler geometry is
to solve the more general conjecture on the existence of cscK metrics.

\begin{conj} [Yau-Tian-Donaldson] The underlying polarized K\"ahler manifold $(M, [\omega])$ is K-stable if and only if there exists a constant scalar curvature K\"ahler metric in $[\omega]$. 
\end{conj}

\subsection{A new continuity path}

 For any K\"ahler manifold $(M, [\omega]),$  consider the space of K\"ahler potentials
 \[
 {\cal H} = \{ \varphi: \omega_\varphi = \omega + \sqrt{-1} \partial \bar \partial \varphi > 0,\;\;{\rm on }\;\;M\}.
 \]
 E. Calabi proposed to study the existence of \emph{extremal K\"ahler metric} in 1982. This
has been a central problem in K\"ahler geometry since its inception.  A K\"ahler metric is called
``extremal'' if the complex gradient vector field of its scalar curvature function is
holomorphic.    A special case is the so-called
\emph{cscK metric} (constant scalar curvature K\"ahler metric) when this vector field vanishes.   To attack the existence problem of cscK metrics,
we propose to study the following continuous path.

 For any positive, closed  $(1,1)$-form $\chi $, define a continuous path
$ t\in [0,1]$ as

\begin{equation}
  t \cdot   \left( R_{\varphi_t} -{{[C_1(M)]\cdot [\omega]^{[n-1]} } \over {[\omega]^{[n]}}} \right) 
  = (1 -  t) \cdot \left(\tr_{\varphi_t}\chi - {{[\chi]\cdot [\omega]^{[n-1]} } \over {[\omega]^{[n]}}}\right).  \label{eq:continouspath1}
\end{equation}
and for simplicity denote
\[
\underline{\chi} = {{[\chi]\cdot [\omega]^{[n-1]} } \over {[\omega]^{[n]}}} \qquad{\rm and}\qquad  \underline{R} = {{[C_1(M)]\cdot [\omega]^{[n-1]} } \over {[\omega]^{[n]}}}.
\]
where $[\omega]^{[k]}=\frac{[\omega]^k}{k!}$. Let $C_t=(1-t)\underline{\chi}-t\underline{R}$, then Equation (\ref{eq:continouspath1}) is rewritten as

\begin{align}
(1-t)\tr_{\varphi_t}\chi-t\;R_{\varphi_t}=C_t.\label{eq:twisted cscK}
\end{align}

\begin{defi}\label{def1-4}  A K\"ahler metric is called twisted cscK metric if its scalar curvature satisfies Equation (\ref{eq:twisted cscK}). We call it twisted extremal K\"ahler metric
if the left hand side of Equation  (\ref{eq:twisted cscK}) gives rise to a holomorphic vector field.
\end{defi}
When $t=1,$ this reduces to the equation for cscK metrics. Let $I$ denote the set of time parameter $t\in [0,1]$ such that Equation (\ref{eq:continouspath1}) 
can be solved at time $t.\;$ As usual,  our goal is to first prove that $I$ is not empty which usually means finding a starting point
where we can solve this equation.   Then, to prove $I$ is open which is crucial for this program to be viable. The hard part is of course
to prove $I$ is closed which involves hard \emph{a priori estimate}. \\

For each $t\in [0,1]$, we call a solution to Equation (\ref{eq:continouspath1})  as \emph{twisted cscK metric} with $(1-t) \chi.\;$ When $t={1\over 2}$, this is exactly the equation considered by Stoppa
\cite{Stoppa09}, which in turns was motivated from J. Fine's work \cite{Fine04}.   The first theorem we proved here is

\begin{theo}[Openness]\label{openness} For any $\chi>0$ and $t \in (0,1)$, if there exists one solution to Equation   (\ref{eq:continouspath1})  for time $t \in (0,1)$, then there exists a small $\delta>0$
such that for any $t'\in (t-\delta, t+\delta) \cap [0,1), $ there exists a solution to Equation  (\ref{eq:continouspath1})  for time $t'.\; $
\end{theo}

This is a crucial result which establishes the validity of the path to solve the cscK metric problem in a given K\"ahler class.  At technical level,  it should be compared with
the well-known 
deformation theorem proved by LeBrun-Simanca \cite{LeBrun-Simanca} where they discussed the deformation of cscK metrics when the K\"ahler class and/or complex structure changes. 
With the \emph{openness} of $I$ at present, the next important matter is wether $I$ is non-empty. In particular, we need a starting point.  \\

When $t=0,$ this reduces to the well-known equation

\begin{equation}
 \tr_\varphi\chi = \underline \chi.  \label{eq:jequation}
\end{equation}
which is the Euler-Lagrange equation for the well-known $J$ functional introduced in \cite{chen00}, defined by the following formula of its derivative:
\begin{equation}
  {{\mathrm{d} J_\chi}\over {\mathrm{d}t}} =  \displaystyle \int_M\; {{\partial \varphi}\over {\partial t}} (\tr_{\varphi}\chi-\underline{\chi})\omega_{\varphi}^{[n]}. \label{eq:jfunctional}
\end{equation}

We will delay discussions of $J$ functional to the next section and for now, we just remark that there are several known obstructions to the solution of Equation (\ref{eq:jequation}). For instance, 
\[
 \underline{\chi}\cdot [\omega]  > [\chi].
\]
The main point is,  if our goal is to attack existence of cscK metrics, we only need to choose \emph{one} $\chi$ for which $I$ is non-empty. For a moment of thought,
it is obvious that we can always solve Equation (\ref{eq:jfunctional}) at $t = 0$ if we choose $\chi $ to be in the same class as $[\omega]$.  Thus, we can always find 
a smooth positive $(1,1)$-form $\chi$ such that set $I$ is non-empty. In light of the openness theorem 1.5,  this leaves the next problem both interesting and important:

\begin{q} \label{q1-6} For any $\chi>0, $  if there exists one solution to Equation   (\ref{eq:continouspath1})  for time $t =0$, then there exists a small $\delta>0$
such that for any $t \in [0, \delta) \cap [0,1), $ there exists a solution to Equation  (\ref{eq:continouspath1})  for time $t.\; $
\end{q}

 We can prove  directly that the linearized operator is strictly elliptic or semi-elliptic for $t\in [0,1], \;$ as we will explain later.  The case at $t=0,1$ are subtle for different reason.
 At time $t=1, $ the equation is  degenerated elliptic where the \emph{kernal}
is induced by the underlying holomorphic vector field. We encourage readers to \cite{chen-Zeng14} for the discussions of this problem.  Likewise, we also defer the discussion of the $t=0$ to a later paper.  The difficulty cause at $t=0\;$ is of different nature: while all strictly elliptic operators,    the inverse operator lost
two derivatives when compairing to $t>0.\;$ Thus,  we need to address this problem more carefully. The immediate thought would be use either a version of Nash-Moser inverse function theorem or the method via adiabatic limit as in a beautiful paper \cite{Fine04}.  However,  there is another different but plausible approach and we would like describe it now.  Heuristically, one can take a flow approach as in \cite{chenhe05}.  One note that at $t=0$, the $J$ flow converges to the solution
$\tr_\varphi \chi = const.\;$  In other words, for a sufficient small $C^{4,\alpha}$ neighborhood of this solution, the $J$ flow is stable (in other words,
for any flow initiate inside this neighborhood will never leave this neighborhood). Intuitively then,  the twisted Calabi flow for $t$ small enough is stable in this neighborhood
 and the argument in \cite{chenhe05} might be adopted to settle this problem.  \\
 
Another interesting twist is to compare this
path with Donaldson's Continuity Path on conical KE metrics (see discussions in Section \ref{sect2}).  The corresponding problem there is to show the existence of conical K\"ahler-Einstein metrics when conical angle is
very small.\footnote{Suppose that $[D] = [\chi] = C_1(M), $ then the path we consider can be reduced to the usual Donaldson's path on conical KE metrics.}  Note that in the case of conical KE path, the existence for small angle is established through difficult work by R. Berman \cite{Berm13} and JMR \cite{JMR11} etc.
It will be interesting to see if the method in \cite{Berm13} can be extend over to our situation. 
 Suppose that there is a sequence of smooth, positive $(1,1)$ form $\chi_\epsilon$ such that $\displaystyle \lim_{\epsilon\rightarrow 0} \chi_\epsilon = D$
and the convergence is smooth away from divisor $D$, then twisted conical cscK metric path \ref{eq:twistedconicalcscKpath}
  and the aformentioned path is naturally
connected thorough this limiting process. It will be interesting to see if one can prove from the existence of equation for $t$ small  to the existence of conial KE with small angle, or vice
versa.   As a matter of fact,  we may generalize a conjecture of Donaldson to include the case of twisted cscK metric:

\begin{conj} Suppose $[D] = [\chi]$ and $D$ is generic,  then the existence of twisted cscK metric at $t \in [0,1]$ is equivalent to the existence
of conical cscK metric at time $t \in [0,1].\;$
\end{conj}

Note that in canonical class, the necessary part is established in \cite{cds12-1} and the sufficient part is open. If we drop the assumption that $D$ is generic, then one expect the sufficient part to be false. It is interesting if we can re-construct Tian-Yau's theorem on complete Calabi-Yau metric on K\"ahler manifold outside a divisor or the existence of conical KE metric with small conical angle via twisted cscK metric approached suggested here. While technically it is still complicated, the conceptual picture is nonetheless very clear. \\

One surprising observation  is that we can avoid this problem of ``jump" or regularity at $t=0$ in Fano manifold by working on a slightly different path (\ref{path3}). \footnote{The path (\ref{path3}) at $[0,1]$ is equivalent to the path  (\ref{eq:continouspath1})  at $[{1\over 2},1]$. }
Note that in Fano manifold, in any K\"ahler class, we can find a metric $\omega$ such that 
\[
Ric\;\omega > 0.
\]
If we choose $\chi = Ric\;\omega> 0$, then

\begin{cor}[Openness]\label{openness} Suppose $\chi $ is a positive Ricci form in any Fano manifold. Then for any  $t \in [0,1)$, if there exists one solution to Equation   (\ref{eq:continouspath1})  for time $t \in [0,1)$, then there exists a small $\delta>0$
such that for any $t'\in [t-\delta, t+\delta) \cap [0,1), $ there exists a solution to Equation  (\ref{path3})  for time $t'.\; $
\end{cor}

Let $E$ be the well-known K-energy functional introduced by T. Mabuchi \cite{Ma87}, whose derivative is given below:
\begin{equation*}
{{\mathrm{d}E}\over {\mathrm{d}t}} = - \int_M\; {{\partial \varphi}\over {\partial t}} (R_{\varphi} - \underline{R}) \omega_\varphi^{[n]}
\end{equation*}
we then call 
$$E_{\chi,t}=(1-t) J_\chi+t E$$
 the \emph{twisted K-energy functional}.

Following Stoppa \cite{Stoppa09}  and Donaldson \cite{Dona02}, we have

\begin{prop}\label{prop1-5} For any $\chi$ and $t \in [0,1]$, $\omega\rightarrow (1-t)\tr_{\varphi_t}\chi-t R_{\varphi_t}-C_t$ can be viewed as 
 a moment map associated with the infinite dimensional group of exact symplectic diffeomorphisms.

\end{prop} 

Following Donaldson \cite{Dona02}, Stoppa \cite{Stoppa09} and in particular, Lejmi-Sz$\acute{\text{e}}$kelyhidi  \cite{Szekelyhidi14}, we can also formulate a notion of \emph{twisted K-stability} which we will delay the discussions to Section \ref{sect5}. In  \cite{chen04}, we proved that $J$ functional is strictly convex along $C^{1,1}$ geodesic segments.
Following the recent work of  Berman-Berndtsson\cite{Ber14-01}, Chen-Li-Paun \cite{ChenLiPaun14}, we arrive at the following expected statement

\begin{prop}\label{prop1-6}
 For any $t \in [0,1]$, the twisted functional $ E_{\chi,t}$ is convex along any $C^{1,1}$ geodesic segment and in particular,
strictly convex when $t < 1.\;$
\end{prop}

When $t=0$, this is due to \cite{chen04}; when $t=1,$ this was first conjectured by the author and proved in recent works \cite{Ber14-01,ChenLiPaun14}.
The proposition above is essentially a combination of these two results (cf. \cite{Ber14-01,ChenLiPaun14}) . As corollary, we proved that
 
 \begin{cor} \label{cor1-7}
 For any $t \in [0,1]$, the functional $ E_{\chi,t}$ is bounded from below if Equation (\ref{eq:continouspath1}) has a solution for $t$.  \end{cor}
 
 \begin{cor} \label{cor1-8}
 For any $t \in [0,1)$, the twisted cscK metric is unique in its K\"ahler class.   \end{cor}
 
 
 We can certainly reformulate the well-known YTD conjecture for the twisted cscK metric. 
 
 \begin{conj} There exists a twisted cscK metric if and only if it is twisted K-stable.
 \end{conj}
 
 We also believe that we should have a corresponding notion of twisted Paul's stability which is equivalent to the existence of twisted cscK metric.  It is well known that 
 the existence of K\"ahler-Einstein metrics implies also Paul's stability if $Aut(M, J)$ is discrete. Conversely,  Paul's stability implies that the K-energy is proper in all finite Bergman spaces \cite{paul13} .  For semi-stable orbits,  we believe the following conjecture holds
 
 \begin{conj} If $(M, [\omega], J)$ is destabilized by  $(M, [\omega], J')$ (where the second one admits a cscK metric). Then for any $\chi \in [\omega],$ there
 exists a unique twisted cscK metric for $ (1-s) R_\varphi - s \tr_\varphi \chi = c_s.\;$
 \end{conj} 
 
\begin{theo}[\cite{chen-Zeng14}] Let $\chi>0$ be any closed, positive $(1,1)$-form in $[\omega].\; $ If there exists a cscK metric,  then for $t\in [0,1]$
close enough to $1$, there exists a solution to Equation  (\ref{eq:continouspath1}).
\end{theo}

We really follow the idea of Bando-Mabuchi to solve the path for $1-t$ small enough.  In their case, they need to solve the \emph{Aubin continuity path} along the way
from $t=1$ to $t=0.\;$ In the present situation, it is technically not feasible to do ``weak compactness" for a family of \emph{twisted cscK metrics} yet. Nonethelese, the problem of weak compactness of twisted cscK metric is one of the fundamental problems
one ultimately need to address (cf. \cite{Tian-Viaclovsky}  \cite{chen-weber} and references therein for discussion on this important
topics).

\subsection{Related problems}\label{sect1-3}

Following our discussions about this one parameter path (Equation (\ref{eq:continouspath1})), it makes sense to define

\begin{defi}\label{def1-11} For any $\chi > 0, $ define $R (\chi)$ to be the supremum of time $t\in [0,1]$ for which this equation is solvable.
\end{defi} 

In view of the \emph{openness theorem}(Theorem \ref{openness}), we make the following conjecture:

\begin{conj}\label{conj1-12}
 For any two closed, positive $(1,1)$-forms $\chi_1, \chi_2$ in the same cohomology class, we have
\[
R(\chi_1) = R(\chi_2).
\]
\end{conj}
 
In the classical Aubin's continuous path,  this conjecture was proved by G. Sz$\acute{\text{e}}$kelyhidi \cite{Sz11}.  More interestingly, together with T. Collins in \cite{CoSz} they also proved that if this path is solvable for $\chi_1 > 0$
at time $t=0$, then it is also solvable for $\chi_2 > 0$ at $t =0$ provided that $[\chi_1] = [\chi_2].\;$ It will be really interesting to answer 

\begin{q} Can we characterize the polarized K\"ahler manifold  $(M, [\omega])$ where there exists a $(1,1)$-form $\chi > 0$ such that $R(\chi)=1$ but with no cscK metrics in $(M,[\omega])?\;$
\end{q}

Another important problem is to study the following \emph{twisted Calabi flow} for any $t\in [0,1]$, where the family of K\"ahler potentials $\varphi=\varphi_s$ evolves by:
\begin{equation}
{{\partial \varphi}\over {\partial s}} =   t R_{\varphi}- (1-t) \tr_\varphi \chi + C_t. \label{eq:twisted Calabi flow}
\end{equation}
It is interesting and important to establish the short time existence for proper initial metrics.  The author believes this is true, along with
a stability theorem at the neighborhood of twisted cscK metric (as in \cite{chenhe05} for the usual Calabi flow).   Note that in Riemann surface \cite{Pook2015}, J. Pook proved short and long time existence
of the twisted Calabi flow where the form depends on time $t>0.\;$  More importantly, one can adopt two important conjectures
about the Calabi flow to the twisted case here.
\begin{conj}[Chen]\label{conjglobalexistence}
 The twisted Calabi flow exists globally for any smooth initial K\"ahler potential.
\end{conj}

Similarly, we can re-formulate  Donaldson's conjectural pictures at the case of Calabi flow to our settings: 
\begin{conj}\label{conj1-14}
 Suppose the twisted Calabi flows have global existence. Then the asymptotic behavior of the twisted Calabi flow starting from 
$(M, \omega, \chi, J)$ falls into the following possibilities:
\begin{enumerate}
\item The flow converges to a twisted cscK metric on the same complex manifold $(M, J)$;
\item The flow converges, up to a differmorphism, to a twisted extremal K\"ahler metric;
\item The manifold does not admit a twisted cscK metric or twisted extremal metric but the transformed flow $(M, [\omega_t], \chi_t, J_t)$  converges to $(Y, [\omega_\infty], \chi', J' )$
which forms a twisted extremal K\"ahler metric, possibly with at least codimension 2 singularities.   
\end{enumerate}
\end{conj}

More interestingly, it will be nice to understand the important result of J. Streets \cite{Streets12}  as well as
more recent exciting result on the Calabi flow \cite{LWZ15} in the twisted Calabi flow case. We believe the result can be made more precise
because of the strict convexity of the twisted K-energy functional.\\

Another important conjecture we want to raise  is

\begin{conj}\label{conjproper} For any $t \in (0,1)$, the twisted functional $E_{\chi,t}$ is proper in terms of geodesic distance function 
if and only if Equation (\ref{eq:continouspath1}) has a solution.
\end{conj}

Note that when $t=1$, this is the conjecture on the properness of the K-energy functional ( see Conjecture \ref{chenpropconj} below) .  We remark that this is another reason why we believe Conjecture \ref{conj1-12} is correct.   In Section \ref{sect3}, we will prove the necessary part of this conjecture. \\
 
{\bf{Acknowledgement}} The author wishes to thank his students Yu Zeng and Chengjian Yao for critical help through the preparation of this paper.  Very recently,  Yoshinori Hashimoto
 informed us that he can solve question \ref{q1-6} with an outline of proof. His proof can also be extended to the case $t\in (0,1).\;$
 
\section{Discussions about various known paths}\label{sect2}

According to E. Calabi, the famous K\"ahler-Einstein problem can be reduced to a problem of complex Monge-Amp$\grave{\text{e}}$re equation:

\begin{equation}\label{eqn2-5}
\det \left( g_{\alpha \b \beta} + {{\partial^2 \varphi}\over {\partial z_\alpha \partial  \bar {z}_\beta}}\right) =  e^{-\varphi + h_\omega} \det (g_{\alpha \b \beta})
\end{equation}
where 
\[
 \omega_g  +\sqrt{-1} \partial \bar \partial \varphi  > 0\;\; {\rm on\; M}.
\]

To attack the existence problem, one adopts a continuous path 

\begin{equation}
\det \left( g_{\alpha \b \beta} + {{\partial^2 \varphi  }\over {\partial z_\alpha \partial  \bar {z}_\beta}}\right) =  e^{- t \varphi + h_\omega} \det (g_{\alpha \b \beta}),\qquad t\in [0,1]. \label{eq:aubinpath1}
\end{equation}

When $t = 0$, this is Calabi's volume conjecture.  The strategy is to prove the set 

$$\{t:  \in [0,1]\mid {\rm Equation\; (\ref{eq:aubinpath1})\; can\;\;be\;solved\;\; at\;\; time} \;\; t\}$$
is both open and closed.  The openness is more or less standard.
To derive closedness, we need to obtain \emph{a priori estimate} which is independent of $t.\;$
The memorial feature of the resolution of the Calabi conjecture by S.-T. Yau is to reduce the second order estimate to a  $C^0$ estimate on K\"ahler potential, which is followed by a $C^3$ estimate of E. Calabi. This is of course well-known to K\"ahler geometers. However, the following might be less well-known.  A cscK metric equation can be decomposed as a coupled second order equations  which we will describe below.

\begin{align}
\log \frac{\det (g_{\alpha \bar\beta} + \frac{\partial^2 \varphi}{\partial z_\alpha \partial  \bar{z}_\beta})}{\det (g_{\alpha\bar\beta} )}& = F, \label{eqn2-7}\\
\Delta_\varphi F & = - \underline{R}  +\tr_\varphi  \text{ Ric}_g\label{eqn2-8}
\end{align}

The following proposition is known:

\begin{prop} \label{prop2-1}
If $g_\varphi$ is a cscK metric and it is quasi isometric to the background metric $g$, then the $C^{4,\alpha}$ norm
on $\varphi$ is uniformly controlled.
\end{prop}
Here is a few words about the proof of this proposition: since $g_\varphi$ is quasi-isometric, then Equation (\ref{eqn2-8}) is
uniformly elliptic with a bounded right hand side. Therefore, by De Giorgi, $[F]_{C^{\alpha}(M,g)}$
is uniformly bounded. Substituting this into Equation (\ref{eqn2-7}),  it becomes a complex Monge-Amp$\grave{\text{e}}$re equation
with $C^\alpha$ bound on the right hand side. According to theory of  Caffarelli, Evans-Krylov and the recent
observation of Y. Wang \cite{yuwang11} (cf. also Chen-Wang\cite{chenywang14-2}), we know $[\varphi]_{C^{2,\alpha}(M,g)}$ is uniformly
controlled. Iterating from Equation (\ref{eqn2-7}) and (\ref{eqn2-8}) once more, we get $[\varphi]_{C^{4,\alpha}(M,g)}$ is uniformly
controlled.\\

Inspired by this proposition, we propose the following naively looking equation:

\begin{align}
\log \frac{\det (g_{\alpha \bar\beta} + \frac{\partial^2 \varphi}{\partial z_\alpha \partial  \bar{z}_\beta})}{\det (g_{\alpha\bar\beta} )}& = F, \label{eqn2-9}\\
\Delta_\varphi F  
& =  t(- \underline{R}  +\tr_\varphi  \text{Ric } g) \label{eqn2-10}
\end{align}

\noindent for any $t \in [0,1].\;$ Set $I$ to be all $t\in [0,1]$ such that the above equation can be solved. Then, obviously $0\in I.\;$ since it reduces to the well-known Calabi's volume conjecture albeit it is trivial in this setting.
The fundamental problem is if the set $I$ necessarily open?   It is hard to tell directly the answer to this crucial question. More importantly, it is quite difficult to imagine the role
of the sign of ``$\text{Ric } g$'' played in this family of equations. Purely from PDE consideration, I think both positivity and negativity of ``$\text{Ric }g$" have their own advantages and pitfalls.
However, these two signs are distinctly different from geometric perspective. In fact, the \emph{openness} depends crucially on the sign of ``$\text{Ric }g$''.  It is open if  it is positive.  This can be readily seen if we re-write this path  of pair of equations  as
\begin{equation}
 R_\varphi = (1-t) \tr_\varphi \text{Ric }g  +  t \underline{R} \label{path3}
\end{equation}
This is very similar to Equation  (\ref{eq:continouspath1}) with $\chi = \text{Ric }g$. Thus, in light of Theorem \ref{openness}, the \emph{openness} holds as long as $\text{Ric }g > 0$.
It is not clear what happens if $\text{Ric }g < 0.\;$\\

On the other hand, we observed that there is an untraced version of this continuous path, which is the well-known \emph{Aubin path}. Note that in K\"ahler-Einstein settings, the Aubin path reads as

\begin{equation}
(1-t)\chi-t \text{ Ric }\omega=\frac{C_t}{n}\omega \label{eq:aubinpath2}
\end{equation}

This casts the ``traditional Aubin path (\ref{eq:aubinpath1})" into new light: that this is also a continuous family of twisted cscK metrics over parameter $t\in [0,1].\; $ It is not difficult
to see
\begin{prop} Along the Aubin path, if $(M,[\omega])$ is K semi-stable, then for every $t \in [0,1)$, it is twisted K-stable.
\end{prop}
Following strategy of \cite{cds12-1, cds12-2,cds12-3}, we may ask 
\begin{q} If there exists a sequence of twisted cscK metric $\omega_i$ for $(M, [\omega],\chi, t_i)$ where $t_i\rightarrow \bar t$ such that 
\begin{enumerate}
\item   $(M, [\omega],\chi, \bar t)$ is twisted K-stable;
\item   Partial $C^0$ estimates holds for $(M, \omega_{i}, \chi, t_i)$,
\end{enumerate}
then does $(M, [\omega],\chi, \bar t)$ admit twisted cscK metric?
\end{q}
According to G. Szekelyhid \cite{Szekelyhidi13},  who in turns follows \cite{cds12-2}, the partial $C^0$ estimate holds for a sequence of K\"ahler metrics over \emph{Aubin path}.   If the answer to aformentioned
question is true, then we can prove Yau's stability conjecture by following Aubin's path.  This leaves a
very interesting question

\begin{q} For any sequence of twisted cscK metric $\omega_i$ for $(M, [\omega],\chi_i, t_i)$, when does the partial $C^0$ estimate hold?
\end{q}
 In the singular case where $\chi=2\pi [D]$ for a divisor $D\subset M$, this equation reduces to
 
\begin{equation}\label{eqn9}
\text{Ric }\omega=-\frac{C_t}{nt}\omega+2\pi\frac{1-t}{t}[D]
\end{equation}
it provides a variant of \emph{Donaldson's continuity path} of conical K\"ahler-Einstein metrics \cite[Equation (27)]{Dona12}. The question is wether the parallel theory holds in this case?

This continuity path could be viewed as a natural generalization of the classical \emph{Aubin's continuity path}
\begin{equation}
\text{Ric }\omega_t=t\omega+(1-t)\chi
\end{equation}
concerning the study of K\"ahler-Einstein metrics. Taking trace of the equation (\ref{eqn9}) gives the notion of ``conical cscK metric'':

\begin{equation}
R_\varphi=-\frac{C_t}{t}+2\pi\frac{1-t}{t}\tr_\varphi [D]
\label{eq:twistedconicalcscKpath}
\end{equation}

We call $\omega$ this \emph{conical constant scalar curvature K\"ahler metric} if it satisfies:
\begin{enumerate}
\item It is quasi isometric to a K\"ahler metric with cone angle $2\pi t$ near the divisor $D$;
\item  The metric tensor is in $C^{\alpha, t}$ near divisor $D\;$ for some $\alpha \in (0, \min(1,{1\over t} -1));\;$
\item The scalar curvature is constant outside the divisor $D$.
\end{enumerate}

We could ask the following\\

\begin{q}  If there is no tangential holomorphic vector in $(M, D), $  can we deform the conical cscK metric with angle $t>0?\;$ 
More interestingly, do we always have a conical cscK metric for $t$ small enough?
\end{q}


\section{Convexity and moment map pictures}\label{sect3}
\subsection{K-energy functional and moment map picture}
One memorial feature is that, via the work of S. Donaldson \cite{Dona96} and Fujiki, the scalar curvature (regarded as the Lie algebra element) of a K\"ahler metric can be viewed as a moment map for an action of the group of exact sympletic diffeomorphisms on the space of all compatible almost complex structures.  In 1998, the author \cite{chen991} established the existence of $C^{1,1}$ geodesic segment between two smooth K\"ahler potentials and proved
that the K-energy functional is convex along $C^{1,1}$ geodesic segment if $C_1 \leq 0.\;$ In general, the best regularity for geodesic segment might be $C^{1,1}$ only. 
At the time, the best we can prove is the following 

\begin{theo}[\cite{chen05, chen-sun09}]Suppose $\varphi_0,\varphi_1$ are two K\"ahler potential and $\varphi_t (t \in [0,1]) $ is the $C^{1,1} $ geodesic segment connecting $\varphi_0$ to $\varphi_1.\;$ Then,

$$(\mathrm{d}E, \varphi_t'|_{t=0})|_{\varphi_0} \leq  (\mathrm{d} E, \varphi_t'|_{t=1})|_{\varphi_1}$$
\end{theo}

However, this theorem leads
to the following conjecture \cite{chen00} by the author. 

\begin{conj} \label{conj3-3}The K-energy functional is convex on any $C^{1,1}$ geodesic segment.
\end{conj}

New understanding in K\"ahler geometry leads to a proof to this full conjecture by the work of Berman-Berndtsson\cite{Ber14-01} and independently
by Chen-Li-Paun \cite{ChenLiPaun14}.  Another related conjecture is the following

\begin{conj}\label{chenpropconj} There exists a cscK metric if and only if the K-energy functional is coercive in terms of geodesic distance to the maximal
invariant, totally geodesic subspace induced by automorphism group.
\end{conj}
This should be seen as a continuation of Conjecture \ref{conjproper}.   In K\"ahler-Einstein manifold without holomophic vector field, Ding-Tian proved that the K-energy functional is proper in terms of Aubin functional. 
The corresponding conjecture (in terms of geodesic distance) has very little progress. Very recenty, 
in Fano manifold with no holomorphic vector field,  T. Darvas \cite{Darvas1401} proved that the Ding functional is proper in terms of $L^1$ geodesic distance. 
This work gives a very good indication that the full conjecture above should be true (module some mild modifications). In toric surface and for $t=1$,  X.-H Zhu and B. Zhou can establish a 
weak existence of extremal K\"ahler metric if the modified K-energy functional is proper \cite{zz08}. This conjecture more likely can be established in toric surface settings first. \\

Now we give a proof of  the necessary part of Conjecture \ref{conjproper} with the assumption that $J_\chi$ is bounded from below.

 \begin{proof}  The following functional is well-known 
  \[
J(\varphi) = \int_M\; \varphi \left(\omega^n -\omega_\varphi^n\right) = 
\int_M\; \sqrt{-1}\p \varphi \wedge \b \p \varphi \wedge \left(\displaystyle \sum_{k=0}^{n-1} \omega^k \wedge \omega_\varphi^{n-k-1} \right) > 0.
\]
According to T. Davas's recent work \cite{Darvas1401,Darvas1402}, this functional is equivalent to geodesic distance in $\cal H$ with respect to Mabuchi's metric.  Recall a
decomposition formula of K-energy in \cite{chen00}: 
For any $\phi \in {\cal H}_\omega$, we have
\[\begin{array}{lcl} E(\varphi) & = &  \displaystyle \int_M\;
\ln {{\omega_\varphi^n}\over {\omega_0^n}}\; \omega^{n}  - J_{Ric\;\omega_0} (\varphi). \label{appl:Jfunc}
\end{array}\]
where $J_{Ric\;\omega_0}$ is the $J_\chi$ functional defined by Formula \ref{eq:jfunctional} for $\chi=Ric\;\omega_0$.


  The first term of this decomposition formula (of K-energy) is called Entropy functional and we will denote it as $E_0(\varphi)$.   We claim that Entropy functional is proper
  in terms of geodesic distance.  According to Davas's recent work \cite{Darvas1401,Darvas1402}, we only need to prove the entropy functionl
  dominates this $J$ functional. This is a well-known fact and we include a short proof here for completeness (cf. \cite{chen05} for this calculation).  
  According G. Tian, there is a positive constant $\alpha > 0$ which depends only
on the K\"ahler class $ [\omega]$ such that for any $\varphi \in \cH$, we have
\[
\int_M\; e^{-\alpha (\varphi -\sup \varphi)} \omega^n \leq C,
\]
or 
\[
\int_M\; e^{-\alpha (\varphi-\sup\varphi) - \log {{\omega_\varphi^n}\over \omega^n}} \omega_\varphi^n \leq C.
\]
If we set 
\[
\int_M\; \omega^n  = 1, 
\]
 we have
\[
\int_M\; -\alpha (\varphi - \sup \varphi) - \log  {{\omega_\varphi^n}\over \omega^n} \omega_\varphi^n \leq C.
\]
Therefore, 
\[\begin{array}{lcl}
\alpha \sup \varphi & \leq & \alpha \int_M\; \varphi \omega_\varphi^n  + \int_M \;  \log  {{\omega_\varphi^n}\over \omega^n} \omega_\varphi^n\\
& \leq &  \alpha \int_M\; \varphi \omega_\varphi^n  + \int_M \;  \log  {{\omega_\varphi^n}\over \omega^n} \omega_\varphi^n.\\
\end{array}
\]
In other words,
\[
\alpha \int_M\; \varphi \omega^n \leq \alpha \sup \varphi  \leq  \alpha \int_M\; \varphi \omega_\varphi^n  + \int_M \;  \log  {{\omega_\varphi^n}\over \omega^n} \omega_\varphi^n.
\]
Consequently,
\[
\alpha J(\varphi) = \alpha \int_M\; \varphi (\omega^n - \omega_\varphi^n) \leq  \int_M \;  \log  {{\omega_\varphi^n}\over \omega^n} \omega_\varphi^n.
\]
Thus, the entropy functional is proper in terms of geodesic distance.  \\

Now, we are ready to prove the necessary part of this theorem.  We assume that 
\[
 \tr_\varphi \chi =  {[\chi]\cdot [\omega]^{[n-1]}\over{[\omega]^{[n]}}}
 \]
 can be solved in $[\omega].\;$   By convexity, we know that 
$J_\chi$ has lower bound in $\cal H$ first. Then, for $\epsilon > 0$ small enough,
we have
\[
J_\chi \geq  \pm \epsilon J_{Ric\;\omega} - C.
\]
This is because we can solve the corresponding equation
\[
 \tr_\varphi (\chi \pm \epsilon Ric\;\omega)  =  {[\chi]\cdot [\omega]^{[n-1]}\over {[\omega]^{[n]}}}  + \epsilon   {[C_1(M) ]\cdot [\omega]^{[n-1]}\over {[\omega]^{[n]}}}
 \]
 via perturbation for small enough $\epsilon > 0.\;$ The desired inequality then follows from convexity of $J_{\chi \pm \epsilon Ric\;\omega}$ functional over $C^{1,1}$ geodesic segment in $\cal H.\;$

Suppose $\omega_\varphi$ is a twisted cscK metric for $\chi > 0$ and $ t_0 \in (0,1), $ then there exists a small $\delta > 0$ such that
the twisted cscK metric equation can be solved for any $t' \in (t_0, t_0 + 2 \delta) \subset (0,1).\;$  Fix $t' = t_0 + {\delta}.\;$  Then,
\begin{equation}
  E +  {{1-t'}\over {t'}} J_\chi \geq - C.
\end{equation}
Note that the coefficient $ {{1-t}\over {t'}}'$ is strictly less than $ {{1-t_0}\over {t_0}}.\;$ In other words, this inequality holds for any number close enough to $ {{1-t_0}\over {t_0}}.\;$
 Now
 \[
\begin{array} {lcl}  t_0 E + (1-t_0) J_\chi & = & (t_0 -\epsilon) E + (1- t_0-\delta') J_\chi  + \epsilon E + \delta' J_\chi\\
& = & (t_0 -\epsilon) (E + {{1-t_0 -\delta' }\over {t_0 -\epsilon}} J_\chi) + \epsilon E_0 + (\delta' J_\chi + \epsilon J_{Ric\;\omega}) + \epsilon \underline{R} I(\varphi).
\end{array}
\]
Now, for fixed $\delta' $ small enough, we can find $\epsilon $ small enough so that
\[
\delta' J_\chi + \epsilon J_{Ric\;\omega} \geq -C
\]
and
\[
E + {{1-t_0 -\delta' }\over {t_0 -\epsilon}} J_\chi \geq -C.
\]
 Thus, we prove
\[
t_0 E + (1-t_0) J_\chi \geq \epsilon E_0 -C
\]
It follows that the twisted K-energy is proper in terms of geodesic distance.
  \end{proof} 
 In summary, we prove
 \begin{theo}\label{thm3-4} Suppose  that 
 \[
 \tr_\varphi \chi =  {[\chi]\cdot [\omega]^{[n-1]}\over {[\omega]^{[n]}}}
 \]
 can be solved in $[\omega].\;$  For any $t \in (0,1)$, the twisted K-energy functional is coercive in terms of
 geodesic distance if one of the following condition holds:
 \begin{enumerate}
 \item There exists a constant scalar curvature metric;
 \item  The K-energy functional is bounded from below;
 \item  There existed a twisted cscK metric for $t \in (0,1).\;$
 \end{enumerate}
 \end{theo}
\begin{conj} For any $C^{1,1}$ K\"ahler potential, we can have a Calabi flow or twisted Calabi flow initiated from this potential and the flow becomes
instantly smooth.  Consequently, the $C^{1,1}$ minimizer of the twisted K-energy is always smooth. 
\end{conj}

In canonical K\"ahler class, both statements has been asserted true by the efforts of many mathematicians.   The weak K\"ahler Ricci flow was initially introduced to obtain partial regularity for $C^{1,1}$ minimizer on the K-energy functional in \cite{chentian005} and subsequently \cite{chending} and \cite{chentianzhang}.  Now, the weak Ricci flow is itself
an intensive subject of study:  there are lots of interesting results on the regularity properties starting from weak data, see \cite{EGZ1, EGZ2, GZ, HSDo, NL} for a partial list and reference therein.


\subsection{ $J$ functional }
In \cite{chen04}, the author introduce the so-called \emph{J-flow} as
\[
{{\partial \varphi}\over {\partial t}} =  \underline{\chi} - \tr_\varphi\; \chi.
\]
This is used to study the lower bound of the K-energy functional on ample K\"ahler manifold.  A striking feature of this $J$ flow is its convexity along any $C^{1,1}$ geodesic segment. 

\begin{prop}[\cite{chen04}] $J$ is a strictly convex functional along any $C^{1,1}$
geodesic. In particular, $J$ has at most one critical point in
${\cal H}.\;$
\end{prop}
For the convenience of readers, we re-produce the proof here. 
\begin{proof} Suppose $\varphi_t$  is a $C^{1.1}$ geodesic. In other
words, $\varphi_t$ is a weak limit of the following continuous
equation as $\epsilon \rightarrow 0 $ (with uniform bounds on the
second mixed derivatives of K\"ahler potentials):
\[  \left({{\partial^2 \varphi}\over{\partial t^2}}  - {1\over 2} \;\mid \nabla {{\partial \varphi}\over{\partial t}} \mid_{\varphi}^2\right)
{{\omega_{\varphi_t}^n}\over {n!}}  = \epsilon \cdot
{{\omega_0^n}\over {n!}} .
\]
Denote $g_t$ as the corresponding K\"ahler metric corresponds to
the K\"ahler form $\omega_{\varphi_t}.\;$  Again, we drop the
dependence of $t$ from $g_t$ for convenience from now on. Recall
the  definition of $J$, we have
\[
   \frac{\mathrm{d}J_\chi}{d\, t} = \displaystyle\;\int_V\; {{\partial \varphi}\over{\partial t}}
   \; (g^{\alpha \overline{\beta}} \; \chi_{\alpha \overline{\beta}}) {{\omega_{\varphi}^n}\over {n!}} .\;
\]
Then (denote $\sigma = g^{\alpha \overline{\beta}} \; \chi_{\alpha \overline{\beta}}$ in the following calculation):
\[\begin{array}{lcl}
  \frac{d^2 J}{d\, t^2}
 &  = & \displaystyle\;\int_V\;\left({{\partial^2 \varphi}\over{\partial t^2}} \sigma -  {{\partial \varphi}\over{\partial t}} g^{\alpha \overline{\beta}} ({{\partial \varphi}\over{\partial t}})_{,\overline{\beta} r} g^{r \overline{\delta}} \chi_{\alpha\overline{\delta}} +{{\partial \varphi}\over{\partial t}} \;\sigma \Delta_g\; {{\partial \varphi}\over{\partial t}}
\right) \; {{\omega_{\varphi}^n}\over {n!}} \\
& = &  \displaystyle\;\int_V\;\left( {{\partial^2
\varphi}\over{\partial t^2}} \sigma - {{\partial
\varphi}\over{\partial t}} g^{\alpha \overline{\beta}}
({{\partial \varphi}\over{\partial t}})_{,\overline{\beta} r}
g^{r \overline{\delta}} \chi_{\alpha\overline{\delta}}\right.
\\ & & \qquad \qquad \qquad-\left. ({{\partial \varphi}\over{\partial t}})_{,r} \sigma g^{r \overline{\delta}} ({{\partial \varphi}\over{\partial t}})_{,\overline{\delta}}
- {{\partial \varphi}\over{\partial t}} g^{\alpha \overline{\beta}} \chi_{\alpha \overline{\beta},\overline{\delta}}
g^{r \overline{\delta}} ({{\partial \varphi}\over{\partial t}})_{,r}  \right) \; {{\omega_{\varphi}^n}\over {n!}} \\
& = &  \displaystyle\;\int_V\;\left( ({{\partial^2 \varphi}\over{\partial t^2}}
- {1\over 2} |\nabla {{\partial \varphi}\over{\partial t}}|^2_{g} ) \sigma
- {{\partial \varphi}\over{\partial t}} g^{\alpha \overline{\beta}} ({{\partial \varphi}\over{\partial t}})_{,\overline{\beta} r}
 g^{r \overline{\delta}} \chi_{\alpha\overline{\delta}} \right. \\
 & & \qquad \qquad \qquad \left.- {{\partial \varphi}\over{\partial t}} \left(g^{\alpha \overline{\beta}} \chi_{\alpha \overline{\delta}} g^{r \overline{\delta}}
\right)_{,\overline{\beta}} ({{\partial \varphi}\over{\partial t}})_{,r} \right) \; {{\omega_{\varphi}^n}\over {n!}} \\
& = &  \displaystyle\;\int_V\;\left( ({{\partial^2 \varphi}\over{\partial t^2}}
- {1\over 2} |\nabla {{\partial \varphi}\over{\partial t}}|^2_{g} )\; ( g^{\alpha \overline{\beta}} \; \chi_{\alpha \overline{\beta}} )
+ ({{\partial \varphi}\over{\partial t}})_{,\overline{\beta}} \left(g^{\alpha \overline{\beta}} \chi_{\alpha \overline{\delta}} g^{r \overline{\delta}}
\right) ({{\partial \varphi}\over{\partial t}})_{,r} \right) \; {{\omega_{\varphi}^n}\over {n!}} \\
& \geq & \displaystyle\;\int_V\;({{\partial \varphi}\over{\partial
t}})_{,\overline{\beta}} \left(g^{\alpha \overline{\beta}}
\chi_{\alpha \overline{\delta}} g^{r \overline{\delta}} \right)
({{\partial \varphi}\over{\partial t}})_{,r} \;
{{\omega_{\varphi}^n}\over {n!}} \geq 0.
\end{array}
\]
The last equality holds along any $C^{1,1}$ geodesic.
\end{proof}

Since its inception, the flow is now well studied. According to  Song-Weinkove \cite{SongBen040} , the necessary and sufficient condition for the flow to converge is that
there exists a form $\omega'\in [\omega]$ such that
\[
(n \underline{\chi}  \omega' - (n-1) \chi) \wedge \omega'^{n-2} > 0.
\]
For more updated work on this subject, we refer readers to Weinkove \cite{Weinkove04},  Song-Weinkove \cite{SongBen040} and Fang-Lai-Song-Weinkove \cite{FLSW14}
for further readings.\\

Following this proposition and  the recent works \cite{Ber14-01,ChenLiPaun14}, we can easily prove that the twisted K-energy functional is convex (Propoisition \ref{prop1-6}), and bounded from below if there is a twisted
cscK metric (Corollary \ref{cor1-7}) and finally can prove the uniqueness of twisted cscK metric for $t < 1.\;$

\section{Deformation}\label{sect4}
In this section, we will prove the openness for deformation of twisted cscK metrics (Theorem \ref{openness}).   
Now we assume $t \in (0,1).\;$ Set

\begin{align*}
\mathcal{H}^{4,\alpha}(M) = \{\varphi \in C^{4,\alpha}(M, \mathbb{R})\ | \omega_{\varphi} = \omega + \sqrt{-1} \partial \bar{\partial} \varphi > 0\}\\
\end{align*}
For any closed positive $(1,1)$-form $\chi$ on $M,\;$ we define a   map
\begin{align*}
F: \mathcal{H}^{4,\alpha}(M) \times [0,1] & \longrightarrow   C^{\alpha}(M, \mathbb{R}) \times  [0,1] \\
    ( \varphi , t) & \longmapsto (t(R_{\varphi}- \underline{R})- (1-t)(\tr_{\varphi} \chi - \underline{\chi}), t)\\
\end{align*}
where $R_{\varphi}$ is the scalar curvature of $\omega_{\varphi}.\;$ The openness theorem $0< t< 1$ is equivalent to the following \footnote{We will deal with the case $t=0$ elsewhere.}.

\begin{theo}\label{thm4-1}
If $F(\varphi_{0}, t_{0}) = (0, t_{0})$ for some $t_0 \in (0, 1)$, then for $t\in [0,1)$ which is  sufficiently close to $t_{0}$, we can find $\varphi=\varphi(t)$ such that $F(\varphi, t) = (0, t)$. 
\end{theo}
Consider the linearization of $F$:
\begin{align*}
\mathcal{D}F|_{(\varphi, t)}: C^{4,\alpha}(M) \times \mathbb{R}& \longrightarrow C^{\alpha}(M)  \times \mathbb{R}\\
(u, s) &\longmapsto    (\mathcal{L}_{(\varphi, t)}u + s(R_{\varphi} - \underline{R}+ \tr_{\varphi} \chi - \underline{\chi}), s)
\end{align*}
where $\mathcal{L}_{(\varphi, t)}$ is the linearized operator of the twisted  scalar curvature function, i.e.
\[ \mathcal{L}_{(\varphi, t)} u = -t\Delta_{\varphi}^2 u - \langle \sqrt{-1} \partial \bar{\partial} u , t\text{Ric}_{\varphi} - (1-t) \chi \rangle_{\varphi}\] 
Set $T = \mathcal{D}F|_{(\varphi_0, t_0)},\;$ then
\begin{lem}\label{lem4-2}
$\mathcal{R}(T)$, which is the range of $T$, is closed. 
\end{lem}
Before we prove this lemma, it is important to note that
\begin{lem}  The kernel of $T$ is one dimensional for any $(\varphi,t)$ which is sufficiently close to $(\varphi_0, t_0)$ in $C^{4,\alpha}(M)\times \mathbb{R}.\;$
\end{lem}\label{lem2}
\begin{proof} For any $(u, s) \in \text{Ker } T, $ without  loss of generality we might assume that
\[
\int_M \; u \omega_\varphi^n  = 0.
\]
By definition, we have $s =0$ and 
\[
{\cal L}_{(\varphi, t)} u  = 0.
\]
Thus,
\begin{align*}
\int u\mathcal{L}_{(\varphi, t)} u\omega_{\varphi}^n 
& = \int \big(-t\Delta_{\varphi}^2 u -t u_{ , \bar{\alpha} \beta} (\text{Ric}_{\varphi})_{\alpha \bar{\beta}}  
+ (1-t) u_{ , \bar{\alpha} \beta} \chi_{\alpha \bar{\beta}} \big)u\omega_{\varphi}^n\\
&= \int \big(-t|u_{,\bar{\alpha} \bar{\beta}}|_{\varphi}^2 - (1-t)u_{,\bar{\alpha}} u_{,\beta} \chi_{\alpha \bar{\beta}} 
+ u_{, \bar{\delta}} (t R_{\varphi}  - (1-t) \tr_{\varphi} \chi)_{,\delta} u \big) \omega_{\varphi}^n.
\end{align*}
It follows that
\begin{align*}
(1-t) \int u_{,\bar{\alpha}} u_{,\beta} \chi_{\alpha \bar{\beta}}  \omega_\varphi^n 
& = - \int t|u_{, \bar{\alpha} \bar{\beta}}|_{\varphi}^2\omega_{\varphi^n}  + \int\; u_{, \bar{\delta}} \big(tR_{\varphi}  
- (1-t)\tr_{\varphi} \chi\big)_{,\delta}u \omega_\varphi^n\\
& \leq - \int t|u_{, \bar{\alpha} \bar{\beta}}|_{\varphi}^2 +  \epsilon (\int u^2 +   \int |\nabla u|^2_\varphi ) \\ 
& \leq -(1-\epsilon) t \int |u_{, \bar{\alpha} \bar{\beta}}|_{\varphi}^2 + \epsilon (C+1) \int |\nabla u|^2_\varphi.
\end{align*}
Here $C$ is the Poincar$\acute{\text{e}}$ constant with respect to the metric $\omega_\varphi$ and $\epsilon$ is controlled by the quantity:
$$\sup_M  |\nabla \big( t R_{\varphi}  - (1-t) \tr_{\varphi} \chi\big)|\;$$
which could be made as small as we want for choosing $(\varphi,t)$ sufficiently close to $(\varphi_0, t_0)$. Thus, for any $t_0 <1$, we can choose
a small neighborhood of $t_0$ such that 
\[
   (1-t) c > \epsilon (1 +C) 
\]
here $ 2 c$ is the lower bound of $\chi$ in terms of $\omega_{\varphi_0}.\;$  Consequently, we have
\[
\int |\nabla u|^2_\varphi = \int |u_{, \bar{\alpha} \bar{\beta}}|_{\varphi}^2  = 0
\] 
which implies $u = 0$.
\end{proof}
Now we return to the proof of Lemma \ref{lem4-2}.
\begin{proof}[Proof of Lemma \ref{lem4-2}] Suppose $(f_i, s_i) \in \mathcal{R}(T)$ such that $f_i$ converges strongly in ${C^{\alpha}(M)}  $ to $f $ and $s_i \rightarrow s$ as $i \rightarrow \infty.\;$  We want to prove that $(f, s) \in \mathcal{R}(T).\;$ By definition,  we may assume that $(f_i, s_i) = T(u_i, s_i)$ where $u_i \in C^{4,\alpha}(M).\;$ 
From the definition of the operator $T$, and since $T(u+C, s) = T(u, s),\; $ we can set
 $$\int  u_i \omega_{\varphi_0} = 0$$
 In other words, 

$$\mathcal{L}_{(\varphi_0, t_0)}u_i + s_i(R_{\varphi_0} - \underline{R}+ \tr_{\varphi_0} \chi - \underline{\chi}) = f_i$$
For any function $u \in C^{4,\alpha}(M), $ we have
\begin{align*}
\int u\mathcal{L}_{(\varphi_0, t_0)} u\omega_{\varphi_0}^n 
& = \int \big(-t_0\Delta_{\varphi_0}^2 u - t_0 u_{ , \bar{\alpha} \beta} (\text{Ric}_{\varphi_0})_{\alpha \bar{\beta}}  
+ (1-t_0) u_{ , \bar{\alpha} \beta} \chi_{\alpha \bar{\beta}} \big)u\omega_{\varphi_0}^n\\
&= \int \big(-t_0 |u_{, \bar{\alpha} \bar{\beta}}|_{\varphi_0}^2 - (1-t_0)u_{,\bar{\alpha}} u_{,\beta} \chi_{\alpha \bar{\beta}} 
+ u_{, \bar{\delta}} (t_0 R_{\varphi_0}  - (1-t_0) \tr_{\varphi_0} \chi)_{,\delta}u \big) \omega_{\varphi_0}^n
\end{align*}
Since $F(\varphi_0, t_0) = (0, t_0) (t_0 < 1) $ and $\chi >0$, we have

\begin{align*}
-\int u\mathcal{L}_{(\varphi_0, t_0)} u\omega_{\varphi_0}^n 
\geq  \epsilon_0 \int |\nabla_{\varphi_0} u|^2_{\varphi_0} \omega_{\varphi_0}^n.
\end{align*}
Therefore, we have
\begin{align*}
\epsilon_0 \int |u_i|^2 \omega_{\varphi_0}^n & \leq C\cdot \epsilon_0 \int |\nabla_{\varphi_0} u_i|^2 \omega_{\varphi_0}^n \\
& \leq C \epsilon  \int \big(s_i(R_{\varphi_0} - \underline{R} + \tr_{\varphi_0}\chi - \underline{\chi} )- f_i\big) u_i \omega_{\varphi_0}^n\\
& \leq \frac{1}{2} \epsilon_0 \int |u_i|^2 \omega_{\varphi_0}^n + C.
\end{align*}
Thus, 
\begin{align*}
\int |u_i|^2 \omega_{\varphi_0}^n \leq C.
\end{align*}
It follows that
\[
 \int |\nabla_{\varphi_0} u_i|^2 \omega_{\varphi_0}^n \leq C.
\]
Let's first consider the case when $0< t_0 <1 $. Given the various bounds above, it's not hard to prove that

\[
 \int |\Delta_{\varphi_0} u_i|^2 \omega_{\varphi_0}^n \leq C
\]
or
\[
  \| u_i\|^2_{W^{2,2}_{\varphi_0}}  < C. 
\]
Now we can re-write the equation for $u_i$ as 
\[
t_0 \Delta_{\varphi_0} (\Delta_{\varphi_0} u_i) = - t_0 u_{ i, \bar{\alpha} \beta} (\text{Ric}_{\varphi_0})_{\alpha \bar{\beta}}  + (1-t_0) u_{i , \bar{\alpha} \beta} \chi_{\alpha \bar{\beta}}   + s_i(R_{\varphi_0} - \underline{R}+ \tr_{\varphi_0} \chi - \underline{\chi}) - f_i.
\]
Note that the right hand side is uniformly bounded in $L^2$ space. Thus, we have
\[
  \| \Delta_{\varphi_0} u_i\|^2_{W^{2,2}_{\varphi_0}}  < C. 
\]
or
\[
  \|  u_i\|^2_{W^{4,2}_{\varphi_0}}  < C. 
\]
Following the standard bootstrapping argument in \emph{elliptic theory}, we have

$$
\|u_i\|_{C^{4,\alpha}(M)} \leq C(\|u_i\|_{L^2(M)}+\|f_i -s_i(R_{\varphi_0} -\underline{R}+\tr_{\varphi_0}\chi-\underline{\chi})\|_{C^{\alpha}(M)}) \leq C
$$
It follows that for any $\alpha' < \alpha$ , we can choose a subsequence $u_i \xrightarrow{C^{4,\alpha'}(M)} u$, as $i \rightarrow \infty$ and $u\in C^{4,\alpha}(M)$. 
Let $(f',s)= T(u, s)$, then $f_i  \xrightarrow{C^{\alpha'}(M)} f'$, as $i \rightarrow \infty.$ Therefore, $(f,s)=(f',s)=T(u,s) \in \mathcal{R}(T).$ So we proved that $\mathcal{R}(T) $ is closed.

\end{proof}

Following the standard theory on linear operators among Hilbert spaces, we have
\begin{cor}  The following decomposition holds:
\[
C^{\alpha}(M) \times \mathbb{R} = \mathcal{R}(T) \oplus \mathbb{R} . 
\]
\end{cor}
\begin{proof}

By definition of the operator $T$, we know that for any $(f,s) \in \mathcal{R}(T), \;$ we have $\int_M f \omega_{\varphi_0}^n = 0.\;$
Set $C^\alpha_0(M) = \{ f\in C^\alpha(M): \int_M f \omega_{\varphi_0}^n = 0 \}.\;$ Then,  $\mathcal{R}(T) \subset  C_0^{\alpha}(M) \times \mathbb{R}.\;$
 The equality holds since $T$ is a linear, self adjoint operator.  So the dimension of \emph{Kernel} is the same as the dimension of \emph{coKernal}. 
 Since $ \dim \text{Ker }T  = 1$, thus $\dim \text{coKer }T = 1.\;$ It follows that 
 $$C^{\alpha}(M) \times \mathbb{R} = \mathcal{R}(T) \oplus \mathbb{R} . $$
\end{proof}

Denote $\tilde{\mathcal{R}}(T)= \pi_{C^{\alpha}(M)}\mathcal{R}(T)= C_0^{\alpha}(M) \times \mathbb{R}.\;$
\begin{proof}[Proof of Theorem \ref{thm4-1}]
Without loss of generality, we can assume $\int_{M}\varphi_0 \omega_{\varphi_0}^n =0$. Consider 
 $F^1(\varphi, t) = \pi_{C^{\alpha}(M)}F(\varphi,t ) $ as the projection of $F$ to the $C^\alpha(M)$ component. Thus
 \[
 F^1(\varphi, t) = t\big(R_{\varphi} - \underline{R} ) - (1-t)(\tr_{\varphi} \chi - \underline{\chi}\big).\]
 Consider the map
\begin{align*}
\Psi: \mathcal{H}^{4,\alpha}(M) \times [0,1] &\longrightarrow (C^\alpha_0(M)  \oplus \mathbb{R})  \times [0,1]\\
(\varphi, t) &\longmapsto (\tilde{\pi}\circ F^1(\varphi , t) + \int \varphi \omega_{\varphi_0}^n,t ),
\end{align*}
where $\tilde{\pi}$ is the projection to $C^\alpha_0(M),\;$ i.e. $\tilde\pi (f) = f - \oint_M \; f \omega_{\varphi_0}^n$  for any function $f \in C^{\alpha}(M)$. Note that 
\[ \oint_M \;F^1(\varphi , t_0) \omega_{\varphi_0}^n  = 0
\]
Thus, corresponding to the variation $\delta\varphi=u$, the variation of $\tilde{\pi}\circ F^1$ at $t =t_0$ is:
\[
\delta \left(\tilde{\pi}\circ F^1\right)(\varphi , t) \mid_{t= t_0} =   \mathcal{L}_{(\varphi_0, t_0)} u + s (R_{\varphi_0} - \underline{R}+ \tr_{\varphi_0} \chi - \underline{\chi}). 
\]
It follows that
\begin{align*}
\mathcal{D}\Psi|_{(\varphi_0, t_0)}(u,s) &= \big(\mathcal{D}F^1|_{\varphi_0, t_0}(u,s) + \int u \omega_{\varphi_0}^n, s\big)\\
&= \big(\mathcal{L}_{(\varphi_0, t_0)} u + s (R_{\varphi_0} - \underline{R}+ \tr_{\varphi_0} \chi - \underline{\chi}) + \int u \omega_{\varphi_0}^n , s\big).
\end{align*}
By our discussions earlier, 
$\mathcal{D}\Psi|_{(\varphi_0, t_0)}:C^{4,\alpha}(M) \times \mathbb{R}\rightarrow (C^\alpha_0(M)  \oplus \mathbb{R})  \times \mathbb{R} $ is bijective. Thus, by inverse function theorem, we can find its inverse 
$$\Psi^{-1}: (C^\alpha_0(M) \oplus \mathbb{R})  \times [0,1] \rightarrow C^{4,\alpha}(M) \times [0,1]$$
near $(0, t_0).\;$ In other words, there exists a small open neighborhood $$V_{0,t_0} \subset C^\alpha(M) \times [0,1] = (C^\alpha_0(M)  \oplus \mathbb{R})  \times [0,1] $$ where $\Psi^{-1}$ is well defined 
 such that  $$\Psi^{-1} (V_{0,t_0}) =  U_{(\varphi_0, t_0)} \subset C^{4,\alpha}(M) \times [0,1].\;$$
 Denote
\begin{align*}
\tilde{F}=F\circ \Psi^{-1}:V_{(0,t_0)}  &\longrightarrow (C^\alpha_0(M) \oplus \mathbb{R})  \times [0,1]\\
(w + a, t) &\longmapsto (w + \int_{M}F^1\circ\Psi^{-1}(w + a , t)\;\omega_{\varphi_0}^n, t ).
\end{align*}
and $${f}(w,a,t) = \int_{M}F^1\circ\Psi^{-1}(w + a , t)\omega_{\varphi_0}^n.$$ 
Then the linearized operator is
\begin{align*}
\mathcal{D}\tilde{\mathcal{F}}|_{(w+a,t)}: (C^\alpha_0(M) \oplus \mathbb{R})  \times \mathbb{R} &\longrightarrow (C^\alpha_0(M)  \oplus \mathbb{R})  \times \mathbb{R}\\
(u+b, s) &\longmapsto (u + (\frac{\partial f}{\partial w}(w,a,t) u +\frac{\partial f}{\partial a}(w,a,t) b + \frac{\partial f}{\partial t}(w,a,t)s) ,s).
\end{align*}
Next we consider its kernel $\text{Ker}(\mathcal{D}\tilde{\mathcal{F}}|_{(w+a,t)}).\;$ Since $\tilde F =  F \circ \Psi^{-1} $ and $\Psi$ is bijection, we need to consider the 
$\dim  \text{Ker}(\mathcal{D}\mathcal{F}|_{(\varphi, t)}).\;$ Clearly, we have
$$\text{Ker}(\mathcal{D}\mathcal{F}|_{(\varphi, t)})= \{(u,0) \in C^{4,\alpha}(M) \times \mathbb{R} | \mathcal{L}_{(\varphi, t)} u= 0\}. $$
Following the proof of the previous lemma, we can easily prove that for any $(u,0) \in \text{Ker}(\mathcal{D}\mathcal{F}|_{(\varphi, t)}), \;$
we have
\[
    u  =  \oint_M\; u \omega^n_{\varphi_0}.
\]
On the other hand, it is clear that 
\[
 \{(C,0) \in C^{4,\alpha}(M) \times \mathbb{R} \} 
 \subset \text{Ker}(\mathcal{D}\mathcal{F}|_{(\varphi, t)}).\]
 
 Therefore, $\dim \text{Ker}(\mathcal{D}\mathcal{F}|_{(\varphi, t)}) = 1.\;$ It follows that
  $\dim_{\mathbb{R}}\text{Ker}(\mathcal{D}\tilde{\mathcal{F}}|_{(w+a,t)})  = 1$ for $(w+a, t)$ sufficiently close to $(0,t_0)$.  We claim that  $\frac{\partial f}{\partial a}(w,a,t) =0.\;$
   Otherwise,  $\dim_{\mathbb{R}}Ker(\mathcal{D}\tilde{\mathcal{F}}|_{(w+a,t)}) =0.\;$ Note for any $( u + b, s) \in \text{Ker}(\mathcal{D}\tilde{\mathcal{F}}|_{(w+a,t)},$ we have
   \[
   (u + (\frac{\partial f}{\partial w}(w,a,t) u +\frac{\partial f}{\partial a}(w,a,t) b + \frac{\partial f}{\partial t}(w,a,t)s) ,s) = (0+0, 0).
   \]
   It follows that $ u = s =0$ and 
   \[
   \frac{\partial f}{\partial a}(w,a,t) b = 0.
   \]
   If $\frac{\partial f}{\partial a}(w,a,t) \neq 0, \;$  then $b =0$. It follows that  $\text{Ker}(\mathcal{D}\tilde{\mathcal{F}})|_{(w+a,t)} = 0.\;$
   This is a contradiction so our claim holds.  It follows that $f(w,a,t) = f(w,t)$. Therefore,

$$\tilde{F} (w+a, t) =  (w + f(w,t), t)$$
We want to find the preimage of $\tilde{F}$ for $(0,t)$. We claim that  $f(0,t)=0.\;$  First, we look at $(\varphi, t) =  \Psi^{-1}(0,a,t).\;$ It means that
\[
\Psi (\varphi, t) = (0,a, t).
\]
In other words, we have
\[
(\tilde \pi \circ F^1(\varphi,t) + \int_M \varphi \omega_{\varphi_0}^n, t) = (0, a, t).
\]
It follows that $\int_M \varphi \omega_{\varphi_0}^n =  a$ and $\tilde \pi \circ F^1 (\varphi,t)= 0.\;$ It follows that 
\begin{align*}
F^1(\varphi, t) - \int_{M} F^1(\varphi, t) \omega_{\varphi_0}^n = 0.
\end{align*}
which implies that $F^1(\varphi,t)(x) \equiv C.\;$ Note that $ \int_M F^1(\varphi,t) \omega_{\varphi}^n = 0 $ by definition. Consequently,
$F^1(\varphi,t)(x) \equiv 0,\;$ and then $$f(0,t) = \int_M F^1(\varphi, t) \omega_{\varphi_0}^n = 0.$$
So
\begin{equation*}
\tilde{F}(0,a,t) = (0 + f(0,t), t) = (0,t). 
\end{equation*}
We therefore have $F(\Psi^{-1}(0,a,t)) =  (0,t)$. This completes the proof.
\end{proof}

\section{Twisted K-stability}\label{sect5}
\subsection{Twisted K-stability}
Corresponding to the twisted K-energy functional, for any \emph{test configuration} $\lambda$ of $M$ (see \cite{SZ13} for a precise definition in the algebraic case when $\Omega$ represents the first Chern class of an holomorphic line bundle $L$, roughly speaking, it means that $M$ is embedded in $\mathbb{P}^{N_k}$ by the linear system $|-kL|$ and $\lambda:\mathbb{C}^*\to GL(N_k,\mathbb{C})$ is a one parameter subgroup ) we could define {twisted Futaki invariant}. Suppose $\lambda(s)=s^A$, define a function on $\mathbb{P}^{N_k}$,
$$h_A=\frac{zAz^*}{zz^*}$$
then we make the following definition:

\begin{defi}[twisted Futaki invariant]\label{def5-1}
$$\text{Fut}_t(\Omega,\chi;\lambda)=(1-t)\text{Fut}(\chi;\lambda)+t \text{Fut}(\Omega;\lambda)$$
where $\text{Fut}(\Omega;\lambda)$ is the algebraically defined \emph{Donaldson-Futaki invariant} of $\lambda$, and


\begin{align*}
\text{Fut}(\chi;\lambda)&=\text{Ch}(\chi;\lambda)-\underline{\chi} \text{Ch}(M;\lambda)\\
&=\lim_{t\to 0} \{\int_{M_s}h_A\lambda(s^{-1})^*\chi\wedge (\frac{1}{k}\omega_{FS}|_{M_s})^{[n-1]}-\underline{\chi}\int_{M_s} h_A (\frac{1}{k}\omega_{FS}|_{M_s})^{[n]}\}\\
&=\lim_{t\to 0} \{\int_{M}\chi\wedge \lambda(s)^*(h_A\omega_s^{[n-1]})-\underline{\chi}\int_{M} \lambda(s)^*(h_A \omega_s^{[n]})\}
\end{align*}
\end{defi}

\begin{rem}
In the special case when the central fiber $M_0=\lim_{t\to 0} \lambda(s)M$ is a normal variety, $\text{Fut}_t(\Omega,\chi;\lambda)$ reduces to the usual \emph{log-Futaki invariant} of the log pair $(M_0,(1-t)D_0)$ (first introduced by \cite{Dona12} for the study of conical K\"ahler-Einstein metrics) where $D_0=\lim_{s\to 0} \lambda(s)D$ for a generic divisor $D$ in the linear system $|L_\chi|$.

If $\chi$ and $\omega$ both belongs to $2\pi C_1(M)$, the \emph{twisted cscK metrics} equation  reduces to the Aubin's continuity path, for which the Futaki's invariant is already explicitly defined , see \cite[Formula (4.10)]{SZ13}.
\end{rem}

\begin{defi}[twisted K-semistable/stable]
A triple $(\Omega,\chi,s)$ as above is called \emph{twisted K-semistable}  if 
$$\text{Fut}_t(\Omega,\chi;\lambda)\leq 0$$
for any \emph{test configuration} $\lambda$; and it is called \emph{twisted K-stable} if 
$$\text{Fut}_t(\Omega,\chi;\lambda)<0$$
for any $\lambda$ nontrivial.
\end{defi}
It follows from the definition that $\text{Fut}(\chi;\lambda)$ depends linearly on $\chi$ for any fixed $\lambda$. It is proved in \cite[Thm 6]{SZ13} that 
$$\text{Fut}(\chi;\lambda)\leq \text{Fut}(D;\lambda)$$
for any divisor $D$ in the linear system of $L_\chi$, and for any fixed \emph{test configuration} $\lambda$ the equality holds for generic element $D$ in this linear system. One consequence is that $\text{Fut}(\chi;\lambda)$ is independent of the particular choice of $\chi$, therefore the notion of \emph{twisted K-stability} for a triple $(\Omega,\chi,s)$ is independent of the choice of $\chi$ in a fixed cohomology class also .
\begin{prop}[Linear Interpolation]
If $(\Omega,\chi,t_0)$ and $(\Omega,\chi,t_1)$ are both twisted K-semistable, and one of them is twisted K-stable, then $(\Omega,\chi,t)$ is twisted K-stable for any $t\in (t_0, t_1)$.
\end{prop}
\begin{proof}
Since $\text{Fut}_t(\Omega,\chi;\lambda)$ depends on $t$ linearly for any fixed \emph{test configuration} $\lambda$, the \emph{twisted K-stability} for the two parameters $t_0$ and $t_1$ implies a preferred sign of $\text{Fut}_t$ for any $t$ in between.
\end{proof}

\begin{conj}
If $(\Omega,\chi,t)$ with $t<1$ is twisted K-stable iff Equation (\ref{eq:continouspath1}) admits a solution for the parameter $t$.
\end{conj}

It should be noticed that the \emph{twisted cscK equation} (Equation (\ref{eq:continouspath1}) for $t=\frac{1}{2}$, the middle point of our continuous path) was already studied by \cite{Fine04, Stoppa09}. Stoppa in \cite[Theorem 1.3]{Stoppa09} gave a cohomological obstruction, \emph{K\"ahler slope stability}, to the existence extending the result of Ross-Thomas\cite{RT} for the un-twisted case. In another direction, \cite[Conjecture 1]{Szekelyhidi14} conjectured a numerical criterion for the existence of solution to Equation  (\ref{eq:jequation})(Equation (\ref{eq:continouspath1}) for $t=0$, the starting point of our continuous path), based on the study of \emph{deformation to the normal cone}, a particular type of test configuration, extensively used in the work \cite{RT}. And recently this conjecture was proved in \cite[Theorem 1.3]{CoSz} for toric manifolds, therefore our conjecture here holds on toric manifolds for the starting point.

\begin{rem}
While we were preparing this note, we noticed the recent work by R. Dervan \cite{Dervan} who introduced also the notion of (uniform) twisted K-stability and proved that the existence of twisted cscK metric implies the (uniform) twisted K-stability by using the lower bound of the twisted Calabi functional.\end{rem}

\section{Twisting cscK metric with higher degree form}
While a $(1,1)$-form $\chi > 0$ is convenient, from PDE point of view, there is no particular reasons to restrict oneselves to this setting only.
For any  integer $ k \in \{1,\cdots, n\}$, consider a closed $(k,k)$-form $\mu_k = \chi^k$, one can define a new functional
\[
{{d\, J_{\mu_k}}\over {d\, t}} =  \int_M\; \dot\varphi (\mu_k\wedge \omega_\varphi^{n-k}-c_k\omega_{\varphi}^n).
\] 
where  the constant
$$c_k = {{ [\mu_k] \cdot [\omega]^{n-k}}\over [\omega]^n}$$

{Remark that the \emph{closedness} of $\mu_k$ guarantee that $J_{\mu_k}$ is well-defined.} Note that this functional has been studied in the literature before (cf. \cite{FLSW14}). The Euler-Lagrange equation is
\[
{{ \mu_k\wedge \omega_\varphi^{n-k}}\over \omega_\varphi^n} = c_k
\]
{ When $k= 0$, this is the well-known $I$ functional. } When $k=1$, then this is just the usual $J$ functional we discussed here. When $k=n$, then $\mu$ is a volume form, for which the Euler-Lagrange equation reduces to 

\[
\omega_\varphi^n=\frac{1}{c_n}\mu_n
\]
which is a standard complex Monge-Amp$\grave{\text{e}}$re equation. In fact, this is precisely the Calabi's volume form conjecture. There is a simple observation for this family of functionals in a special case, where $\mu_k=\omega_0^k, \;k=1, \cdots, n$ for some K\"ahler metric $\omega_0$.

\begin{prop}\label{prop6-1}
For any $\phi\in \cH_{\omega_0}$, 
$$J_{\omega_0^n}(\phi)\geq J_{\omega_0^{n-1}}(\phi)\geq \cdots\geq J_{\omega_0}(\phi)\geq\frac{1}{n+1}J(\phi)$$
\end{prop}

\begin{proof}
If suffices to prove $J_{\omega_0^{k+1}}(\phi)\geq J_{\omega_0^k}(\phi)$. However, this is almost obvious by the definition. Take the standard linear path $\phi_t=t\phi$, then

\begin{align*}
J_{{\omega_0^{k}}}(\phi)&=\int_0^1 dt\int_M \dot\phi_t (\omega_0^k-\omega_{\phi_t}^k)\omega_{\phi_t}^{n-k}\\
&=\int_0^1 dt\int_M \phi\; t(\omega_0-\omega_\phi)(\omega_0^{k-1}+\omega_0^{k-2}\omega_{t\phi}+\cdots+\omega_{t\phi}^{k-1})\omega_{t\phi}^{n-k}\\
&=\int_0^1 t\; dt\int_M \sqrt{-1}\partial\phi\wedge\bar\partial \phi\wedge (\omega_0^{k-1}\omega_{t\phi}^{n-k}+\omega_0^{k-2}\omega_{t\phi}^{n-k+1}+
\cdots +\omega_{t\phi}^{n-1})
\end{align*}

Since each integrand is a positive term,  $J_{\omega_0^{k+1}}\geq J_{\omega_0^k}$. And the last inequality $J_{\omega_0}\geq\frac{1}{n+1}J$ is well-known (cf. Bando-Mabuchi\cite[Eq. 1.6.4]{Bando87}). 

\end{proof}

Similar to the case $k=1$, there is also a moment map
picture associated with this functional. \\

One can also define a new continuity path as

\[
t (R_\varphi - \underline{R}) = (1-t) ({{ \mu_k\wedge \omega_\varphi^{n-k}}\over \omega_\varphi^n}-  c_k),\qquad \forall t \in [0,1]
\]
This is the Euler-Lagrange equation of the twisted K-energy functional
\[
E_{\mu_k,t}= t E + (1-t) J_{\mu_k}.
\]
Then, one can formulate twisted K-stability and other geometric notions similarly.  It is straightforward to prove the following
\begin{theo} For any closed positive $\chi>0$, the twisted K-energy functional  $E_{\mu_k,t}$ is convex along $C^{1,1}$ geodesic segment for any $k =1,2,\cdots ,n.\;$
\end{theo}
Following proof in Section \ref{sect3}, we have
\begin{theo} For any $ k > 0, $  the twisted cscK metric equation is open for any $ t \in [0,1).\;$
\end{theo}
\begin{theo}  If $(M, [\omega])$ is K-semistable and K-stable for $J_{\mu_k}$, then the twisted trip structure $(M, [\omega],\mu_k, t)$ is K-stable for $t \in [0,1).\;$
\end{theo}
Moreover, we can ask question parallel to those in previous sections. To avoid redundancy, we only list the following conjecture

\begin{conj} There exists a twisted cscK metric if and only if the corresponding twisted K-energy functional is proper in terms of geodesic distance.
\end{conj}
Here we give a proof of the necessary part for $k=n$ and $t \in (0,1).\;$
\begin{proof} Suppose $\omega_\varphi$ is a twisted cscK metric for $\mu > 0$ an positive volume form and $ t_0 \in (0,1), $ then there exists a small $\delta > 0$ such that
the twisted cscK metric equation can be solved for any $t' \in (t_0, t_0 + 2 \delta) \subset (0,1).\;$  Fix $t' = t_0 + {\delta}.\;$  Then,
\[
 t' E + (1-t') J_\mu.
\]
is bounded from below. 

Firstly, by Yau's solution of Calabi's Volume form Conjecture, there exists a unique K\"ahler metric $\omega_0\in [\omega]$ such that

$$\mu=c_n\omega_0^n$$
which implies $J_{\mu}=c_n J_{\omega_0^n}\geq \frac{c_n}{n+1}J$ by Prop. \ref{prop6-1}.



Now
\[
\begin{array} {lcl}  t_0 E + (1-t_0) J_\mu & = & {t_0\over t'} ( t' E + (1-t_0) {t'\over t_0} J_\mu)\\
& =& {t_0\over t'} ( t' E + (1-t') J_\mu)  +  {t_0\over t'} ((1-t_0) {t'\over t_0} - (1- t') ) J_\mu) \\
& = &  {t_0\over t'} ( t' E  + (1-t') J_\mu)  +  {t_0\over t'} ( {t'\over t_0} -1) J_\mu \\
& =&    {t_0\over t'} ( t' E  + (1-t') J_\mu)  + { {t' - t_0}\over t'} J_\mu\\
& \geq &  {t_0\over t'} ( t' E  + (1-t') J_\mu)  + { {\delta}\over {\delta+t_0}} J_\mu .
\end{array}
\]
Now the first term is bounded from below and the second is proper, thus the twisted K-energy $t_0 E + (1-t_0) J_\mu $ is proper in terms of the geodesic distance.
\end{proof}


\newpage

\begin{thebibliography}{05}

\em\bibitem{ArezzoPacard06}
C.  Arezzo, F. Pacard,
\newblock  \it Blowing up and desingularizing constant scalar curvature K\"ahler manifolds. 
\newblock Acta Math. 196 (2006), no. 2, 179-228.

\bibitem{Are-Pac-Singer11}
C. Arezzo, F. Pacard and M. Singer:
\newblock \it Extremal metrics on blowups. 
\newblock {Duke Math. J. 157 (2011), no. 1, 1-51.}


\bibitem{Bando87}
S. Bando and T.Mabuchi:
\newblock \it {U}niquness of {E}instein {K}\"ahler metrics modulo connected
  group actions.
\newblock {In {\em Algebraic {G}eometry}, Advanced Studies in Pure Math., 1987.}


\bibitem{Bedford76}
E.D. Bedford and T.A. Taylor:
\newblock \it The {D}irichlet problem for the complex {M}onge-{A}mp$\grave{\text{e}}$re operator.
\newblock {\em Invent. Math.}, 37:1--44, 1976.

\bibitem{Berm12-05}
R.J. Berman:
\newblock \it K-polystability of $\mathbb{Q}$-Fano varieties admitting K\"ahler-Einstein metrics. 
\newblock Preprint, \url{http://arxiv.org/abs/1205.6214}.


\bibitem{Berm13}
R. J. Berman:
\newblock  A thermodynamical formalism for Monge-Amp\'e re equations, Moser-Trudinger inequalities and K\"ahler Einstein metrics,
\newblock  Adv. Math. 248 (2013), 1254?1297.

\bibitem{Ber14-01}
R.J. Berman, B. Berndtsson:
\newblock Convexity of the K-energy on the space of K\"ahler metrics.
\newblock Preprint, 	\url{http://arxiv.org/abs/1405.0401}

 \bibitem{Blocki99}
 Z. Blocki:
 \newblock On the regularity of the complex Monge-Amp$\grave{\text{e}}$re operator,
 \newblock  Contemporary Mathematics 222, Complex Geometric Analysis in Pohang, ed. K.-T.Kim, S.G.Krantz, pp.181-189, Amer. Math. Soc. 1999.
 
  \bibitem{Block}
 Z. Blocki:
 \newblock Interior Regiularity of the complex Monge-Amp$\grave{\text{e}}$re equation in complex domains
 \newblock Duke Math. Jour., vol 105, no. 1, 167-181, 2000.
 



\bibitem{Caf} L.A. Caffarelli:
\newblock {Interior $W^{2,p}$ estimates for solutions to the Monge-Amp\'ere equation},
\newblock Annals of Math. 131, 135-150, 1990

\bibitem{CG} L.A. Caffarelli and C.E. Guti\'errez:
\newblock {\it Properties of the solutions of the linearised
Monge-Amp\'ere equation} ,
\newblock Amer. Jour. Math. 119 423-465 1997

\bibitem{CNS84}
 L.A.~Caffarelli, L.~Nirenberg and J.~Spruck:
\newblock The Dirichlet problem for nonlinear second-order elliptic equation
  {I}. Monge-Amp$\grave{\text{e}}$re equation.
\newblock {\em Comm. on pure and appl. math.}, XXXVII:369--402, 1984.

 
\bibitem{CKNS85}
L.A.~Caffarelli, J.J. Kohn, L.~Nirenberg,   and J.~Spruck:
\newblock The Dirichlet problem for nonlinear second-order elliptic equation
  {II}. complex Monge-Amp$\grave{\text{e}}$re equation, and Uniformly Elliptic equations.
\newblock {\em Comm. on Pure and Appl. math.}, 38:209--252, 1985.


\bibitem{calabi82}
E.~Calabi:
\newblock Extremal {K}\"ahler metrics.
\newblock In {\em Seminar on Differential Geometry}, volume~16 of {\em 102},
  pages 259--290. Ann. of Math. Studies, University Press, 1982.

\bibitem{calabi85}
E.~Calabi:
\newblock Extremal {K}\"ahler metrics, {II}.
\newblock In {\em Differential geometry and Complex analysis}, pages 96--114.
  Springer, 1985.


\bibitem{ChenLiPaun14}
X.-X. Chen, L.  Li and  M. Paun:
\newblock Approximation of weak geodesics and subharmonicity of Mabuchi energy
\newblock Preprint, 	\url{http://arxiv.org/abs/1409.7896}.

\bibitem{chen991}
X.-X. Chen:
\newblock Space of {K}\"ahler metrics.
\newblock {\em Journal of Differential Geometry}, 56(2):189--234, 2000.


\bibitem{chen04}
X.-X. Chen:
\newblock A new parabolic flow in K\"ahler manifolds.
\newblock  Comm. Anal. Geom. 12 (2004), no. 4, 837?852. 

\bibitem{chen00}
X.-X. Chen:
\newblock On the lower bound of the {M}abuchi energy and its application.
\newblock  {\em Internat. Math. Res. Notices} 2000, no. 12, 607--623. 

\bibitem{chentian005}
X.-X. Chen and G.~Tian:
\newblock Foliation by holomorphic discs and its application in K\"ahler
  geometry, 2003.
\newblock  Publ. Math. Inst. Hautes �tudes Sci. No. 107 (2008), 1--107. 

\bibitem{chen05}
X.-X. Chen:
\newblock   {\it Space of K\"ahler metric (III)---Lower bound of the Calabi energy}.
\newblock Invent. Math. 175 (2009), no. 3, 453?503.

\bibitem{chen08}
X.-X. Chen:
\newblock   {\it Space of K\"ahler metric (IV)---On the lower bound of the K-energy}.
\newblock \url{http://arxiv.org/abs/0809.4081}.

\bibitem{chending} 
X.-X. Chen and W.-Y. Ding 
\newblock {\it Ricci flow on surfaces with degenerate initial metrics.}
\newblock J. Partial Differential Equations 20 (2007), no. 3, 193?202.


\bibitem{chentianzhang}
X.-X. Chen. G. Tian, Z. Zhang: 
\newblock  {\it On the weak K\"ahler-Ricci flow.}
\newblock  Trans. Amer. Math. Soc. 363 (2011), no. 6, 2849?2863.

\bibitem{cds12-1}
X.-X. Chen,  S. Donaldson and S. Sun:
\newblock   {\it K\"ahler-Einstein metrics on Fano manifolds. I: Approximation of metrics with cone singularities.}.
\newblock J. Amer. Math. Soc. 28 (2015), pp. 183-197 (I).
\bibitem{cds12-2}
X.-X. Chen,  S. Donaldson and S. Sun:
\newblock   {\it K\"ahler-Einstein metrics on Fano manifolds. II: Limits with cone angle less than $2\pi$}.
\newblock J. Amer. Math. Soc. 28 (2015), pp. 199-234.

\bibitem{cds12-3}
X.-X. Chen,  S. Donaldson and S. Sun:
\newblock   {\it K\"ahler-Einstein metrics on Fano manifolds. III: Limits as cone angle approaches $2\pi$ and completion of the main proof.}
\newblock J. Amer. Math. Soc. 28 (2015), pp. 235-278.


\bibitem{chenhe05}
X.-X. Chen and W-Y. He.
\newblock On the Calabi flow,
\newblock  Amer. J.  Math. 130 (2008), no. 2, 539--570.

\bibitem{CLW08}
X.-X. Chen, C. LeBrun and B. Weber:
\newblock {\it On conformally K�hler, Einstein manifolds}. 
\newblock J. Amer. Math. Soc. 21 (2008), no. 4, 1137?1168

\bibitem{chen-sun09}
X.-X. Chen and S. Sun:
\newblock   {\it Space of K\"ahler metrics (V)---K\"ahler quantization}.
\newblock  Metric and differential geometry, 19?41, Progr. Math., 297, Birkh�user/Springer, Basel, 2012.


\bibitem{chen-sun10}
X.-X. Chen and S. Sun:
\newblock   {\it Calabi flow, Geodesic rays, and uniqueness of constant scalar curvature K\"ahler metrics}.
\newblock Ann. of Math. (2) 180 (2014), no. 2, 407-454.



\bibitem{chenwang09-1}
 X.-X. Chen and B. Wang:
 \newblock {\it Space of Ricci flows (I)}.
\newblock  Comm. Pure Appl. Math. 65 (2012), no. 10, 1399-1457. 

 \bibitem{chenwang14-1}
 X.-X. Chen and B. Wang:
\newblock {\it Space of Ricci flows (II)}.
\newblock Preprint,    \url{http://arxiv.org/abs/1405.6797}, submitted.

\bibitem{chenwang09-2}
 X.-X. Chen and B. Wang:
 \newblock {\it K\"ahler Ricci flow on Fano manifolds(I)}.
\newblock J. Eur. Math. Soc. (JEMS) 14 (2012), no. 6, 2001-2038.


\bibitem{chenywang13}
X.-X. Chen and Y.-Q. Wang:
\newblock {\it Bessel Functions, Heat Kernel and the Conical K\"ahler-Ricci Flow} ,
\newblock Preprint,	\url{http://arxiv.org/abs/1305.0255}.


\bibitem{chenywang14-1}
X.-X. Chen and Y.-Q. Wang:
\newblock {\it On the long time behaviour of the Conical K\"ahler- Ricci flows} ,
\newblock Preprint,	\url{http://arxiv.org/abs/1402.6689}.


\bibitem{chenywang14-2}
X.-X. Chen and Y.-Q. Wang:
\newblock {\it On the regularity problem of complex Monge-Amp$\grave{\text{e}}$re equations with conical singularities} ,
\newblock Preprint,	\url{http://arxiv.org/abs/1405.1021}.

\bibitem{chen-weber}
X.-X. Chen and B. Weber:
\newblock  {\it Moduli spaces of critical Riemannian metrics with $L^{n/2}$ norm curvature bounds}.
\newblock   Adv. Math. 226 (2011), no. 2, 1307-1330.

\bibitem{chen-Zeng14}
X.-X. Chen, M. Paun and Y. Zeng:
\newblock  On deformation of extremal metrics,
\newblock 	\url{http://arxiv.org/abs/1506.01290}.
\bibitem{Semmes93}
R.R. Coifman and S.~Semmes:
\newblock Interpolation of {B}anach Spaces, {P}erron Processes, and {Y}ang-{M}ills.
\newblock {\em Amer. J. Math.}, 115(2):243--278, 1993.

\bibitem{CoSz}
T.Collins, G.Sz$\acute{\text{e}}$kelyhidi: 
 { Convergence of the J-Flow on Toric manifolds}. Preprint,
 \url{http://arxiv.org/abs/1412.4809}
 
\bibitem{Darvas1401}
T. Darvas:
\newblock Envelopes and Geodesics in Spaces of K\"ahler Potentials. Preprint,
\newblock  \url{http://arxiv.org/abs/1401.7318}.

\bibitem{Darvas1402}
T. Darvas:
\newblock The Mabuchi Geometry of Finite Energy Classes. Preprint,
\newblock \url{http://arxiv.org/abs/1409.2072}.

\bibitem{Dervan}
R. Dervan,
\newblock  Uniform Stability of twisted constant scalar curvature K\"ahler metrics. Preprint,
\newblock  \url{http://arxiv.org/abs/1412.0648}


\bibitem{HSDo}
H. S. Do:
\newblock{ Degenerate complex Monge-Amp$\grave{\text{e}}$re flows on strictly pseudoconvex domains,}
\newblock \url{http://arxiv.org/abs/1501.07167}


\bibitem{Dona96}
S.~K. Donaldson:
\newblock Symmetric spaces, {K}{\"a}hler geometry and {H}amiltonian dynamics.
\newblock {\em Amer. Math. Soc. Transl. Ser. 2, 196}, pages 13--33, 1999.
\newblock Northern California Symplectic Geometry Seminar.

\bibitem{Dona02}
S.~K. Donaldson:
\newblock  Scalar curvature and projective embeddings. I.
\newblock  J. Differential Geom. 59 (2001), no. 3, 479--522.

\bibitem{D1} S.~K. Donaldson:
\newblock { Scalar curvature and stability of toric varieties}, 
\newblock Jour. Differential Geometry 62, 289--349, 2002

\bibitem{D3}S.~K. Donaldson:
\newblock {Extremal metrics on toric surfaces: a continuity method}, 
\newblock J. Differential Geom. 79 (2008), no. 3, 389--432.

\bibitem{D4}S.~K. Donaldson:
\newblock { Constant scalar curvature metrics on toric surfaces},
\newblock Geom. Funct. Anal. 19 (2009), no. 1, 83--136.


\bibitem{Dona052}
S. K. Donaldson:
\newblock Lower bounds on the Calabi functional.
\newblock  J. Differential Geom. 70 (2005), no. 3, 453--472. 

\bibitem{Dona12}
S. K. Donaldson;
\newblock {K\"ahler metrics with cone singularities along a divisor}, 
\newblock Essays in Mathematics and its Applications. 2012, 49-79.


\bibitem{EGZ1}
P. Eyssidieux, V. Guedj, A. Zeriahi
\newblock Weak solutions to degenerate complex Monge-Amp$\grave{\text{e}}$re Flows I,
\newblock {\em Math. Annal. DOI
10.1007/s00208-014-1141-4}

\bibitem{EGZ2}
P. Eyssidieux, V. Guedj, A. Zeriahi
\newblock {Weak solutions to degenerate complex Monge-Amp$\grave{\text{e}}$re Flows II}
\newblock {\url{http://arxiv.org/abs/1407.2504}}


\bibitem{FLSW14}
H.  Fang, M. Lai , J. Song and B. Weinkove:
\newblock The J-flow on K\"ahler surfaces: a boundary case. 
\newblock Anal. PDE 7 (2014), no. 1, 215-226. 

\bibitem{Fine04}
J. Fine:
\newblock  Constant scalar curvature K\"ahler metrics on fibred complex surfaces.
\newblock  J. Differential Geom. 68 (2004), no. 3, 397-432.


\bibitem{Futaki83}
A. Futaki:
\newblock An obstruction to the existence of Einstein K\"ahler 
metrics.
\newblock  {\em Invent. Math.} 73 (1983), no. 3, 437-443.

\bibitem{GuanB98}
B. Guan:
\newblock The {D}irichlet problem for complex {M}onge-{A}mp$\grave{\text{e}}$re equations and
  regularity of the plui-complex Green function.
\newblock {\em Comm. {A}na. {G}eom.}, 6(4):687--703, 1998.


\bibitem{GZ}
V. Guedj, A. Zeriahi:
\newblock {Regularizing properties of the twisted K\"ahler-Ricci flow}.
\newblock {\url{http://arxiv.org/abs/1306.4089}}.

\bibitem{JMR11}
T. D. Jeffres, R. Mazzeo, and Y. Rubinstein, 
\newblock {\it K\"ahler-Einstein metrics with edge singularities}. \url{http://arxiv.org/abs/1105.5216v3}.

\bibitem{Pook2015}
J. Pook:
\newblock Twisted Calabi flow on Riemann surfaces,
\newblock IMRN, p. 1-26, volume 2015, Issue 9.


\bibitem{LeBrun-Simanca} 
C. LeBrun, S. R. Simanca:
\newblock {\it Extremal K\"ahler
metrics and Complex Deformation Theory}, 
\newblock Geom. and Fun. Analysis, Vol. 4, No. 3 (1994), 298-336.

\bibitem{Lebrun-singer}
C. LeBrun, M. Singer:
\newblock  {\it Existence and deformation theory for scalar-flat K\"ahler metrics on compact complex
surfaces},
\newblock Invent. Math.  112  (1993),  no. 2, 273--313.

\bibitem{Szekelyhidi14}
M. Lejmi and G. Sz$\acute{\text{e}}$kelyhidi:
\newblock {\it The J-flow and stability},
\newblock Preprint, 	\url{http://arxiv.org/abs/1309.2821}.


\bibitem{LWZ15}
H.-Z. Li, B. Wang and K. Zheng:
\newblock Regularity scales and convergence of the Calabi flow. Preprint,
\newblock \url{http://arxiv.org/abs/1501.01851}.

\bibitem{Ma87}
T.~Mabuchi:
\newblock Some symplectic geometry on compact {K}\"ahler manifolds {I}.
\newblock {\em Osaka, {J}. {M}ath.}, 24:227--252, 1987.


\bibitem{NL}
E. D. Nezza, C.-H. Lu:
\newblock {Uniqueness and short time regularity of the weak K\"ahler-Ricci flow,}
\newblock{\url{http://arxiv.org/abs/1411.7958}}.


  \bibitem{paul13}
S. Paul:
\newblock Stable Pairs and Coercive Estimates for The Mabuchi Functional,
\newblock Preprint, 	\url{http://arxiv.org/abs/1308.4377}.

  \bibitem{PT06}
S. Paul, G. Tian:
\newblock CM stability and the generalized Futaki invariant II.
\newblock Preprint, 	\url{http://arxiv.org/abs/math/0606505}.

\bibitem{P6}
Pogorelov, A. V.:
\newblock The Drichelet problem for the multidimensional analogue of the Monge-Amp$\grave{\text{e}}$re equation (Russian),
\newblock  Dokl. Akad. Nauk SSSR 201 (1971), 790-793.

\bibitem{P8}
Pogorelov, A. V.:
\newblock  The Minkowski multidimensional problem,
\newblock J. Wiley, New York, 1978

\bibitem{RT}
J. Ross, R. Thomas:
\newblock An obstruction to the existence of constant scalar curvature K\"ahler metrics.
\newblock J. Differential Geom. 72(2006) no. 3, 429--466.

\bibitem{YZelditch}
Y. A. Rubinstein, S.Zelditch:
\newblock Bergman approximations of harmonic maps into the space of K\"ahler metrics on toric varieties
\newblock Priprint,  \url{http://arxiv.org/abs/0803.1249}.


\bibitem{Semmes92}
S.~Semmes:
\newblock Complex {M}onge-{A}mp\`ere equations and sympletic manifolds.
\newblock {\em Amer. J. Math.}, 114:495--550, 1992.





\bibitem{SongBen040}
 J. Song and B. Weinkove:
\newblock On the convergence and singularities of the J-flow with applications to the Mabuchi energy
\newblock  Comm. Pure Appl. Math. 61 (2008), no. 2, 210--229.

\bibitem{stopp08}
J. Stoppa:
\newblock K-stability of constant scalar curvature K\"ahler manifolds.
\newblock  Adv. Math. 221 (2009), no. 4, 1397-1408.


\bibitem{Stoppa09}
 J.~Stoppa:
 \newblock  Twisted constant scalar curvature K\"ahler metrics and K\"ahler slope stability. 
 \newblock J. Differential Geom. 83 (2009), no. 3, 663-691. 
 
 \bibitem{Streets12}
 J. Streets:
 \newblock Long time existence of Minimizing Movement solutions of Calabi flow,
 \newblock 	\url{http://arxiv.org/abs/1208.2718}.
 
\bibitem{Sz11}
G. Sz$\acute{\text{e}}$kelyhidi:
\newblock {Greatest Lower bound on the Ricci curvature of Fano manifolds}, 
\newblock Compositio Math. 147(2011), 319-331.

\bibitem{Szekelyhidi13}
G. Sz$\acute{\text{e}}$kelyhidi:
\newblock The partial $C^0$-estimate along the continuity method,
\newblock Preprint, 	\url{http://arxiv.org/abs/1310.8471}.

\bibitem{SZ13}
{G. Sz$\acute{\text{e}}$kelyhidi:}
\newblock {A remark on Conical K\"ahler-Einstein metrics}, 
\newblock Math. Res. Lett. 20 (2013), no. 3, 581-590.

\bibitem{tian90} 
G. Tian:
\newblock On Calabi's conjecture for complex
 surfaces with positive first Chern class, 
 \newblock Invent. Math. 101, no.1, 101-172.
  

\bibitem{Tian-Viaclovsky} 
G. Tian, J. Viaclovsky:
\newblock {\it Moduli spaces of critical metrics in dimension four}, 
\newblock Adv. Math., 196 (2005),346-372.


\bibitem{tian97}
G. Tian:
\newblock K\"ahler-Einstein metrics with positive scalar curvature. 
\newblock {Invent. Math.} 130 (1997), no. 1, 1--37. 


\bibitem{TY} 
G. Tian, S.-T. Yau: 
\newblock {\it K\"ahler-Einstein metrics on complex surfaces with $C\sb
1>0$}, 
\newblock Comm. Math. Phys.  112  (1987),  no. 1, 175--203.


 \bibitem{tian07}
G. Tian,  X.-H. Zhu:
\newblock  Convergence of K\"ahler-Ricci flow.
\newblock J. Amer. Math. Soc. 20 (2007), no. 3, 675--699.

\bibitem{Tru-Wang}
N. S. Trudinger and X-J. Wang:
\newblock The Monge-Amp$\grave{\text{e}}$re equation and its geometric applications,
\newblock Handbook of Geometric analysis, International Press, 2008, Vol. I, pp. 467-524.

\bibitem{U3} 
Urbas, J. I. E.:
\newblock Regularity of generalized solutions of Monge-Amp$\grave{\text{e}}$re equations, Math. Z. 197 (1988), 365-393.


\bibitem{wang12}
B. Wang:
\newblock On the conditions to extend Ricci flow (II).
\newblock Int. Math. Res. Not. IMRN 2012, no. 14 (2012): 3192?223.




\bibitem{ywang12}
Y.-Q. Wang:
\newblock Pseudolocality of the Ricci flow under integral bound of curvature.
\newblock J. Geom. Anal. 23 (2013), no. 1, 1-23. 

\bibitem{yuwang11}
Y. Wang:
\newblock A remark on $C^{2,\alpha}$-regularity of the complex Monge-Amp$\grave{\text{e}}$re Equation,
\newblock Preprint, \url{http://arxiv.org/abs/1111.0902}.

\bibitem{Weinkove04}
 B. Weinkove:
 \newblock Convergence of the J-flow on K\"ahler surfaces. 
 \newblock Comm. Anal. Geom. 12 (2004), no. 4, 949-965. 
 
\bibitem{Yau78}
S.-T. Yau.
\newblock On the {R}icci curvature of a compact {K}\"ahler manifold and the
  complex Monge-Amp$\grave{\text{e}}$re equation, ${I}^*$.
\newblock {\em Comm. Pure Appl. Math.}, 31:339--441, 1978.


\bibitem{tian98}
G.~Tian.
\newblock Some aspects of {K}\"ahler {G}eometry, 1997.
\newblock Lecture note taken by {M}eike {A}keveld.

\bibitem{zz08}
B. Zhu and X.-H. Zhu
\newblock Minimizing weak solutions for Calabi's extremal metrics on toric manifolds. 
\newblock Calc. Var. Partial Differential Equations 32 (2008), no. 2, 191-217.
 
\end{thebibliography}
\end{document}